\documentclass[10pt]{amsart}

\newif\ifarxiv
\arxivtrue

\usepackage[T1]{fontenc}
\usepackage{lmodern}
\usepackage{amsmath}
\usepackage{amssymb}
\usepackage{amsthm}
\usepackage{mathtools}
\usepackage{mathrsfs}
\usepackage{xcolor}
\usepackage{tikz}
\PassOptionsToPackage{hyphens}{url}
\usepackage{hyperref}
\usepackage{microtype}
\usepackage[textwidth=400pt,textheight=640pt,hcentering,heightrounded]{geometry}

\hypersetup{
  colorlinks=true,
  linkcolor=blue!55!black,
  citecolor=green!40!black,
  urlcolor=blue!65!black
}
\mathtoolsset{showonlyrefs=false}
\allowdisplaybreaks
\newtheorem{theorem}{Theorem}[section]
\newtheorem{proposition}[theorem]{Proposition}
\newtheorem{lemma}[theorem]{Lemma}
\newtheorem{corollary}[theorem]{Corollary}

\theoremstyle{definition}

\theoremstyle{remark}
\newtheorem{remark}[theorem]{Remark}
\newtheorem{fact}[theorem]{Fact}

\makeatletter
\def\subsection{\@startsection{subsection}{2}%
  \z@{.5\linespacing\@plus.7\linespacing}{.3\linespacing}%
  {\normalfont\bfseries}}
\makeatother

\newcommand{\C}{\mathbb C}
\newcommand{\Q}{\mathbb Q}
\newcommand{\Z}{\mathbb Z}
\newcommand{\N}{\mathbb N}
\newcommand{\R}{\mathbb R}
\newcommand{\bbeta}{\boldsymbol{\beta}}
\newcommand{\bv}{\boldsymbol{v}}
\newcommand{\bu}{\boldsymbol{u}}
\newcommand{\bg}{\boldsymbol{g}}
\newcommand{\bh}{\boldsymbol{h}}
\newcommand{\bs}{\boldsymbol{s}}
\newcommand{\bx}{\boldsymbol{x}}
\newcommand{\bpartial}{\boldsymbol{\partial}}
\newcommand{\nsupp}{\operatorname{nsupp}}
\newcommand{\inw}{\operatorname{in}_{\boldsymbol w}}
\newcommand{\Span}{\operatorname{Span}_{\C}}
\newcommand{\Hilb}{\operatorname{Hilb}}
\newcommand{\rank}{\operatorname{rank}}
\newcommand{\coker}{\operatorname{coker}}
\newcommand{\tp}[1]{\prescript{t}{}{#1}}
\newcommand{\Tor}{\operatorname{Tor}}
\newcommand{\Rhat}{\widehat R}
\newcommand{\Shat}{\widehat S}
\newcommand{\cC}{\mathcal C}
\newcommand{\cR}{\mathcal R}
\newcommand{\cF}{\mathcal F}
\newcommand{\cG}{\mathcal G}
\newcommand{\cH}{\mathcal H}
\newcommand{\cK}{\mathcal K}
\newcommand{\cN}{\mathcal N}
\newcommand{\cW}{\mathcal W}
\newcommand{\cZ}{\mathcal Z}
\newcommand{\sS}{\mathscr S}
\newcommand{\fD}{\mathfrak D}
\newcommand{\fE}{\mathfrak E}
\newcommand{\fF}{\mathfrak F}
\newcommand{\fH}{\mathfrak H}
\newcommand{\fK}{\mathfrak K}
\newcommand{\fT}{\mathfrak T}
\newcommand{\fC}{\mathfrak C}
\newcommand{\fj}{\mathfrak j}
\newcommand{\fb}{\mathfrak b}
\newcommand{\fc}{\mathfrak c}
\newcommand{\fa}{\mathfrak a}
\newcommand{\fm}{\mathfrak m}
\newcommand{\fp}{\mathfrak p}
\newcommand{\ba}{\boldsymbol{a}}
\newcommand{\bb}{\boldsymbol{b}}
\newcommand{\bc}{\boldsymbol{c}}
\newcommand{\bp}{\boldsymbol{p}}
\newcommand{\bq}{\boldsymbol{q}}
\newcommand{\bt}{\boldsymbol{t}}
\newcommand{\bw}{\boldsymbol{w}}
\newcommand{\bz}{\boldsymbol{z}}
\newcommand{\balpha}{\boldsymbol{\alpha}}
\newcommand{\bzero}{\boldsymbol{0}}
\newcommand{\bone}{\boldsymbol{1}}
\newcommand{\Ann}{\operatorname{Ann}}
\newcommand{\im}{\operatorname{im}}
\newcommand{\Supp}{\operatorname{Supp}}
\newcommand{\amb}{\mathrm{amb}}
\newcommand{\intr}{\mathrm{intr}}
\newcommand{\mov}{\mathrm{mov}}

\hypersetup{
  pdftitle={The obstruction to intrinsic perturbation for A-hypergeometric series},
  pdfauthor={NAKANO Ryunosuke},
  pdfsubject={The obstruction to intrinsic perturbation for A-hypergeometric series. Mathematics Subject Classification 2020: 33C70 (Primary); 13D07, 13P10, 13F55, 14M25 (Secondary)},
  pdfkeywords={A-hypergeometric system, Frobenius method, intrinsic perturbation, fake exponent, Nilsson series, Groebner basis, affine semigroup, colon ideal, syzygy, hyperplane arrangement, complete intersection}
}

\title[The obstruction to intrinsic perturbation]{The obstruction to
intrinsic perturbation for \(A\)-hypergeometric series}
\author{NAKANO Ryunosuke}
\address{Graduate School of Science, Hokkaido University, Sapporo 060-0810, Japan}
\email{nakano.ryunosuke.i3@elms.hokudai.ac.jp}
\date{}


\begin{document}

\begin{abstract}
		We show that, at a fake exponent of an \(A\)-hypergeometric system and for an ordered negative support family, intrinsic perturbation within \(\ker_{\Z}(A)\) can produce a strictly smaller coefficient space than ambient perturbation, which answers a question of Okuyama--Saito.
	We measure the failure by an intrinsic-perturbation obstruction module, a quotient of two colon ideals whose graded dual is the ambient coefficient space modulo the intrinsic one, and we realize this module by two right-exact sequences involving \(\Tor_1\) and present it finitely by two antichains.
	We give configurations of lattice rank 1 and 2 with nonzero obstruction, including one with an ordered distinguished collection consisting of all negative supports attained only by lattice shifts of nonnegative weight, and thereby settle the case left open by Okuyama--Saito.
	We give, for every integer \(q\geq1\), a configuration of lattice rank 2 with a connected column matroid and with a normal affine semigroup generated by the reduced Gr\"obner basis vectors, whose obstruction module has dimension \(q^2\), so this dimension is unbounded at fixed lattice rank.

\end{abstract}

\maketitle

\section{Introduction}

The Frobenius method produces logarithmic solutions of a resonant differential system by perturbing an exponent and differentiating with respect to the resulting perturbation parameters.
For \(A\)-hypergeometric systems, this method is developed by Saito and by Okuyama--Saito; see \cite{Sai20,OS22}.
Let \(A\in\Z^{d\times n}\) be the matrix of the system, put \(L=\ker_{\Z}(A)\) and \(r=\rank L\), and write \(\bt=\tp{(t_1,\ldots,t_n)}\).
Okuyama--Saito \cite{OS25} fix a fake exponent \(\bv\) and an ordered negative support family \(\cN\), construct the ambient coefficient space, and compare it with the coefficient space obtained by intrinsic perturbation within \(L\); their Theorem~5.5 identifies the ambient coefficient space with the full coefficient space, and their Proposition~6.2 characterizes the equality of the intrinsic and ambient coefficient spaces by an equality of two colon ideals.
Question~7.3 of \cite{OS25} asks whether that equality can fail, and the discussion there records that the problem remains open even for the set that we call the distinguished collection \(\cN_{\bv}\).
In this paper, we answer this question affirmatively, both for a general ordered negative support family and for the distinguished collection, and we measure the failure by an intrinsic-perturbation obstruction module.

We work in \(\Rhat=\C[[t_1,\ldots,t_n]]\) with the ideal \(U=\langle A\bt\rangle\), the monomial ideals \(M_{\cN}\) and \(P_{\cN}(\bt)\) of \eqref{eq:obstruction-M} and \eqref{eq:obstruction-P}, and the monomial \(e=\bt^{\nsupp(\bv)\setminus K_{\cN}}\), where \(\nsupp(\bv)\) is the set of indices \(j\) with \(v_j\) a negative integer and \(K_{\cN}\) is the intersection of the members of \(\cN\).
We define the intrinsic-perturbation obstruction module by
\[
	\fD_{\cN}(e)
	=\frac{(U+P_{\cN}(\bt)):e}
	{(UM_{\cN}+P_{\cN}(\bt)):e}.
\]
Theorem~\ref{thm:lattice-obstruction} identifies this module with the classes represented by multiples of \(e\) in a quotient of \(\Tor_1\), and shows that for an ordered negative support family the module \(\fD_{\cN}(e)\) vanishes precisely when the intrinsic coefficient space equals the ambient coefficient space.
In the lattice coordinates, Theorem~\ref{thm:finite-boundary-classification} presents \(\fD_{\cN}(e)\) by two antichains as a graded quotient \(J_{\cN}^{\intr}/J_{\cN}^{\amb}\) of two ideals in \(r\) variables, and Proposition~\ref{prop:defect-coefficient-duality} identifies the graded dual of that quotient with the ambient coefficient space modulo the intrinsic one, so the Hilbert series of \(J_{\cN}^{\intr}/J_{\cN}^{\amb}\) records the missing coefficients degree by degree.

Fix a generic weight \(\bw\), and write \(\cC(\bw)\subseteq L\) for the affine semigroup generated by the reduced Gr\"obner basis vectors of the toric ideal of \(A\) in the direction \(\bw\).
Theorem~\ref{thm:lattice-counterexample} answers Question~7.3 by a five-column configuration of lattice rank 2 for which \(\cC(\bw)\) is normal, and Theorem~\ref{thm:ordered-distinguished-counterexample} settles the distinguished-collection case left open in \cite{OS25} by an example that fails two of the three hypotheses of the vanishing criterion in Theorem~\ref{thm:rank-two-moving-binary-vanishing}.
Neither normality of \(\cC(\bw)\) nor connectivity of the column matroid forces equality: Theorem~\ref{thm:three-connected-obstruction} gives a six-column configuration whose column matroid is three-connected, and Theorem~\ref{thm:fixed-rank-unbounded} gives, for every integer \(q\geq1\), a configuration of lattice rank 2, with a connected column matroid and normal \(\cC(\bw)\), whose obstruction module has dimension \(q^2\).
In a companion paper, we use these obstruction modules to compute, inside the solution space of a homogeneous system, the codimension of the span of the series obtained by intrinsic perturbation.

We use Theorem~5.5 of \cite{OS25}, and the regular-sequence criteria and bounds below correspond to Proposition~6.6 and Corollary~6.7 there.
Nagamine \cite{Nag26} computes the Hilbert series of the local fake indicial inverse system itself, where the Artinian object is a quotient of a Stanley--Reisner ring.
The quotient studied here measures the failure of intrinsic perturbation to realize the ambient coefficient space and is a quotient of two colon ideals.
\ifarxiv
The ancillary scripts use exact arithmetic on specified inputs to check the successive colon ideals and the Hilbert series identities for two orders of the factors of \(e\), the formulas for \(J_{\cN}^{\amb}\) and \(J_{\cN}^{\intr}\), the colon ideals in the five- and six-column examples, and the Hilbert series of the obstruction modules for \(q=1,2,3,4\) in Section~\ref{sec:unbounded-families}.
\fi

\section{Fake exponents and negative supports}
\label{sec:fake-exponents}

In this section, we recall the \(A\)-hypergeometric system, the negative support families, and the perturbation constructions.

\subsection{The \texorpdfstring{\(A\)}{A}-hypergeometric system}

Let \(A=[\ba_1,\ldots,\ba_n]\in\Z^{d\times n}\) have rank \(d\), and assume throughout this paper that the columns of \(A\) lie in an affine hyperplane not containing the origin.
Set \(L=\ker_{\Z}(A)\) and \(r=\rank L=n-d\).
For \(\bu\in L\), define \(|\bu|\) by \(|\bu|=\sum_{j=1}^n u_j\).
Throughout, \(\N=\{0,1,2,\ldots\}\).

The toric ideal is \(I_A= \left\langle \bpartial^{\bu_+} -\bpartial^{\bu_-} \ \middle|\ \bu\in L \right\rangle \subseteq\C[\partial_1,\ldots,\partial_n]\), where \(\bu=\bu_+-\bu_-\) has disjoint nonnegative parts.
For \(\bbeta\in\C^d\), the \(A\)-hypergeometric ideal \(H_A(\bbeta)\) is generated in the Weyl algebra \(D\) by \(I_A\) and the Euler operators \(E_i-\beta_i =\sum_{j=1}^n a_{ij}x_j\partial_j-\beta_i\) for \(i=1,\ldots,d\).
Here \(\bx=\tp{(x_1,\ldots,x_n)}\) and \(\bpartial=\tp{(\partial_1,\ldots,\partial_n)}\), and we write \(M_A(\bbeta)=D/H_A(\bbeta)\).
For the foundational construction and the fake-exponent formalism, see \cite{GKZ89,GKZ94,Sai02,SST00}.

Fix a generic weight \(\bw\in\R^n\).
Here and below, genericity means that the initial ideal \(\inw(I_A)\) is monomial and that \(\bw\) lies in the interior of a full-dimensional cone of the small Gr\"obner fan of \(H_A(\bbeta)\), so the initial ideal \(\operatorname{in}_{(-\bw,\bw)}(H_A(\bbeta))\) in the Weyl algebra is constant on that interior; see \cite{SST00}.
Every full-dimensional Gr\"obner cone of \(I_A\) contains such weights.
Each wall of these fans is the locus where two monomials of the same \(A\)-degree acquire equal \(\bw\)-weight, and the difference of their exponents is a nonzero element of \(L\).
A weight is therefore generic if its values on the nonzero elements of \(L\) are all nonzero, since such a weight lies on no wall.
The column matroid of \(A\) is the matroid represented over \(\Q\) by the columns of \(A\).
Let \(\cG_{\bw} = \left\{ \bpartial^{\bg^{(i)}_+} -\bpartial^{\bg^{(i)}_-} \ \middle|\ i=1,\ldots,g_{\bw} \right\}\) be the reduced Gr\"obner basis of \(I_A\), written so that in each binomial the first monomial is the initial one.
Thus \(\bw\cdot\bg^{(i)}>0\), where \(\bg^{(i)}=\bg^{(i)}_+-\bg^{(i)}_-\).
Define the affine semigroup associated with \(\bw\) by \(\cC(\bw) =\sum_{i=1}^{g_{\bw}}\N\bg^{(i)}\subseteq L\).
We say that \(\cC(\bw)\) is normal if \(\cC(\bw)\) contains every element of the group generated by \(\cC(\bw)\) that lies in the real cone spanned by \(\cC(\bw)\).
This semigroup lies in \(L\), so its normality is not the normality of the affine semigroup generated by the columns of \(A\).

For \(\bv\in\C^n\), define \(\nsupp(\bv)=\{j\mid v_j\in\Z_{<0}\}\).
A vector \(\bv\) is a fake exponent of \(H_A(\bbeta)\) in the direction \(\bw\) if \(A\bv=\bbeta\) and \(\bv\) is a zero of the fake indicial ideal of \cite[Section~3.2]{SST00}.
By \cite[Corollary~3.2.3]{SST00}, a vector \(\bv\) with \(A\bv=\bbeta\) is a fake exponent if and only if there is a standard pair \((\ba,\sigma)\) of \(\inw(I_A)\) such that \(v_j=a_j\) for \(j\notin\sigma\).
We use this equivalence repeatedly.

\subsection{Negative support families}

Fix a fake exponent \(\bv\), and, for \(\bu\in L\), put \(I_{\bu}=\nsupp(\bv+\bu)\), \(I_{\bzero}=\nsupp(\bv)\), and \(\sS(\bv)=\{I_{\bu}\mid\bu\in L\}\).
For \(I\in\sS(\bv)\), we call \(\cF_I(\bv)=\{\bu\in L\mid I_{\bu}=I\}\) the support fiber of \(I\), and we write \(\cF_I\) when \(\bv\) is fixed.
We call the set defined in \cite[Section~3]{OS25} the distinguished collection and write
\[
	\cN_{\bv}
	=
	\left\{
	I\in\sS(\bv)
	\ \middle|\
	\bw\cdot\bu\geq0
	\text{ for every }\bu\in\cF_I
	\right\}.
\]
Since \(\bv\) is a fake exponent, \(I_{\bzero}\in\cN_{\bv}\); see the discussion preceding \cite[Definition~3.1]{OS25}.
A negative support family for \(\bv\) is a subset \(\cN\) satisfying \(I_{\bzero}\in\cN\subseteq\cN_{\bv}\).
In the terminology of \cite[Definition~3.1]{OS25}, a negative support family \(\cN\) is ordered if, for \(I\in\cN\) and \(J\in\sS(\bv)\), the inclusion \(J\subseteq I\) implies \(J\in\cN\).
We do not assume that the distinguished collection \(\cN_{\bv}\) is ordered.
Put \(K_{\cN}=\bigcap_{I\in\cN}I\).

\subsection{The perturbation constructions}

Choose a Gale dual \(B=(\bb^{(1)},\ldots,\bb^{(r)})\) of \(A\) whose columns form a \(\Z\)-basis of \(L\), and let \(\bs=\tp{(s_1,\ldots,s_r)}\).
For \(\bz\in\C^n\) and \(\bp\in\N^n\), define the falling factorial by \([\bz]_{\bp} = \prod_{j=1}^n z_j(z_j-1)\cdots(z_j-p_j+1)\).
For \(\bu\in L\), set
\[
	a_{\bu}(\bs)
	=
	\frac{[\bv+B\bs]_{\bu_-}}
	{[\bv+B\bs+\bu]_{\bu_+}}.
\]
For a negative support family \(\cN\), put \(m_{\bv,\cN}(\bs) = \prod_{j\in I_{\bzero}\setminus K_{\cN}}(B\bs)_j\), and define the intrinsic perturbation series by
\[
	\Psi_{\cN}(\bx,\bs)
	=
	m_{\bv,\cN}(\bs)
	\sum_{\substack{\bu\in L\\I_{\bu}\in\cN}}
	a_{\bu}(\bs)\bx^{\bv+B\bs+\bu}.
\]
The full coefficient space at \(\bv\) for \(\cN\) is the space of logarithmic polynomials that occur as the coefficient of \(\bx^{\bv}\) in a formal series solution of \(H_A(\bbeta)\) in the direction \(\bw\) all of whose negative supports \(I_{\bu}\) lie in \(\cN\).
If \(\cN\) is ordered, the intrinsic coefficient space at \(\bv\) for \(\cN\) is the space of coefficients of \(\bx^{\bv}\) obtained from \(\Psi_{\cN}(\bx,\bs)\) by the construction in \cite[Theorem~3.2]{OS25}.
Under the same hypothesis, the ambient coefficient space at \(\bv\) for \(\cN\) is the space of coefficients of \(\bx^{\bv}\) obtained by the ambient perturbation construction in \cite[Theorem~5.5]{OS25}.
Theorem~5.5 of \cite{OS25} identifies the ambient coefficient space with the full coefficient space, and Proposition~6.2 of \cite{OS25} compares the intrinsic and ambient coefficient spaces.
The index set of \(m_{\bv,\cN}\) is independent of \(B\), but the polynomial expression of \(m_{\bv,\cN}\) depends on \(B\).
If \(B'=BT\) with \(T\in\operatorname{GL}_r(\Z)\), then the substitution \(f(\bs)\mapsto f(T\bs)\) carries the expression for \(B\) to the expression for \(B'\).

For each \(i\), put \(G^{(i)}(\bv) =I_{-\bg^{(i)}}\setminus I_{\bzero}\), and define
\[
	P_B(\bv)
	=
	\left\langle
	\prod_{j\in G^{(i)}(\bv)}(B\bs)_j
	\ \middle|\
	i=1,\ldots,g_{\bw}
	\right\rangle
	\subseteq\C[\bs].
\]
For \(f(\bs)\in\C[\bs]\) and \(q(\bpartial_{\bs})\in\C[\bpartial_{\bs}]\), put
\[
	\left\langle f,q(\bpartial_{\bs})\right\rangle_{\bs}
	=
	\left.q(\bpartial_{\bs})f(\bs)\right|_{\bs=\bzero}.
\]
For a homogeneous ideal \(I\subseteq\C[\bs]\), the inverse system is
\[
	I^\perp
	=
	\left\{
	q(\bpartial_{\bs})
	\ \middle|\
	\left\langle f,q(\bpartial_{\bs})\right\rangle_{\bs}=0
	\text{ for every }f\in I
	\right\}.
\]
We use this construction for \(P_B(\bv)\) and, in Section~\ref{subsec:finite-boundary-presentation}, for the ideals \(J_{\cN}^{\amb}\) and \(J_{\cN}^{\intr}\).

The \(A\)-hypergeometric system and the fake exponents are taken from \cite{GKZ89,GKZ94,Sai02,SST00}; the negative support families, the perturbation constructions, and the ideal \(P_B(\bv)\) are taken from \cite{Sai20,OS22,OS25}.
The affine semigroup \(\cC(\bw)\), generated by the reduced Gr\"obner basis vectors, also occurs in \cite{Nag26}.

\begin{remark}
	\label{rem:homogeneity-scope}
	The standing assumption on the columns of \(A\) is equivalent to the existence of a linear functional \(\bh\) with \(\bh\ba_j=1\) for every \(j\), and this assumption makes \(I_A\) homogeneous for the standard grading.
	The algebraic part of this paper does not use that grading.
	The following results use only \(\rank(A)=d\) and the standard gradings of \(R_0\) and \(S_0\), introduced in Section~\ref{subsec:finite-boundary-presentation}, in which every entry of \(A\bt\) and every linear form \((B\bs)_j\) has degree one: the obstruction module, the two right-exact sequences, the finite presentation by two antichains, the localization and completion arguments, the Artinian conclusion for the distinguished collection, the Cohen--Macaulay criterion and its refinement at each factor, and the local obstruction computations and vanishing results in low lattice rank.
	The assumption enters through the results of \cite{OS25}, which are stated for a homogeneous configuration and supply the identification of the ambient coefficient space with the full coefficient space.
	We therefore keep the assumption throughout rather than tracking it statement by statement.
\end{remark}

\section{The obstruction to intrinsic perturbation}
\label{sec:lattice-obstruction}

Okuyama--Saito's ambient perturbation construction allows perturbation outside \(L\) and realizes the full coefficient space for an ordered negative support family \cite[Theorem~5.5]{OS25}.
Their intrinsic perturbation construction gives a subspace of the ambient coefficient space, and Proposition~6.2 of \cite{OS25} characterizes the equality of the intrinsic and ambient coefficient spaces by an equality of two colon ideals.
In this section, we form the quotient of these two colon ideals, prove that, for an ordered negative support family, this quotient is the graded dual of the quotient of the ambient coefficient space by the intrinsic coefficient space, and give a finite presentation of the quotient of the colon ideals.

Table~\ref{tab:obstruction-notation} lists the notation used in this section.
\begin{table}[ht]
	\centering
	\small
	\begin{tabular}{@{}p{.29\linewidth}p{.65\linewidth}@{}}
		\(\sS(\bv)\)                            &
		the set of realized negative supports \(I_{\bu}\)                                                   \\
		\(\cN_{\bv}\)                           &
		the distinguished collection                                                                        \\
		\(\cN\)                                 &
		a negative support family                                                                           \\
		\(\Rhat\), \(R_0\)                      &
		\(\C[[t_1,\ldots,t_n]]\) and \(\C[t_1,\ldots,t_n]\), respectively                                   \\
		\(\Shat\), \(S_0\)                      &
		\(\C[[s_1,\ldots,s_r]]\) and \(\C[s_1,\ldots,s_r]\), respectively                                   \\
		\(K_{\cN}\), \(E\), \(e\)               &
		\(\bigcap_{I\in\cN}I\),
		\(I_{\bzero}\setminus K_{\cN}\), and
		\(\bt^E\), respectively                                                                             \\
		\(M_{\cN}\)                             &
		the monomial ideal in \eqref{eq:obstruction-M}                                                      \\
		\(P_{\cN}(\bt)\)                        &
		the monomial ideal in \eqref{eq:obstruction-P}                                                      \\
		\(\cG_{\cN}\), \(\cH_{\cN}\)            &
		the antichains in \eqref{eq:boundary-G-antichain} and \eqref{eq:boundary-H-antichain}, respectively \\
		\(Q_{\cN}(\bt)\)                        &
		\(UM_{\cN}+P_{\cN}(\bt)\)                                                                           \\
		\(\fD_{\cN}(e)\)                        &
		the intrinsic-perturbation obstruction module                                                       \\
		\(\fT_{\cN}\)                           &
		the module in \eqref{eq:boundary-obstruction}                                                       \\
		\(J_{\cN}^{\amb}\), \(J_{\cN}^{\intr}\) &
		the two colon ideals in \eqref{eq:boundary-comparison-ideals}                                       \\
		\(\Gamma_{\cN}\), \(\epsilon_{\cN}\)    &
		the matrix whose kernel gives \(J_{\cN}^{\amb}\) in \eqref{eq:boundary-matrix-kernel} and the class in \eqref{eq:distinguished-boundary-class}, respectively
	\end{tabular}
	\caption{Notation for the obstruction module.}
	\label{tab:obstruction-notation}
\end{table}

\subsection{The obstruction module}

Fix a fake exponent \(\bv\) and a negative support family \(\cN\) as in Section~\ref{sec:fake-exponents}.
The algebraic constructions below do not require \(\cN\) to be ordered.
The additional hypothesis that \(\cN\) is ordered is used for the interpretation in terms of the ambient and intrinsic coefficient spaces.
The obstruction depends on \(\cN\); it is not an invariant of \(A\), \(\bw\), and \(\bv\) alone.
Put \(\bt^J=\prod_{j\in J}t_j\), with \(\bt^{\varnothing}=1\), and \(e=\bt^{I_{\bzero}\setminus K_{\cN}}\).
In the complete local ring \(\Rhat=\C[[t_1,\ldots,t_n]]\), write \(\bt=\tp{(t_1,\ldots,t_n)}\) and define
\begin{align}
	U&=\langle A\bt\rangle,
	\label{eq:obstruction-U}\\
	M_{\cN}
	&=\left\langle
	\bt^{I\setminus K_{\cN}}
	\ \middle|\ I\in\cN
	\right\rangle,
	\label{eq:obstruction-M}\\
	P_{\cN}(\bt)
	&=\left\langle
	\bt^{(I\cup J)\setminus K_{\cN}}
	\ \middle|\
	I\in\cN,\ J\in\sS(\bv)\setminus\cN
	\right\rangle,
	\label{eq:obstruction-P}\\
	Q_{\cN}(\bt)
	&=UM_{\cN}+P_{\cN}(\bt).
	\label{eq:obstruction-Q}
\end{align}
Let \(\Shat=\C[[s_1,\ldots,s_r]]\), and consider the surjection
\[
	\Phi:\Rhat\longrightarrow\Shat, \qquad t_j\longmapsto(B\bs)_j.
\]
The kernel of \(\Phi\) is \(U\).
Write \(\Pi_{\cN}=\Phi(P_{\cN}(\bt))\) and \(m_{\bv,\cN}=\Phi(e)\).

\begin{fact}
	\label{fact:ambient-realization}
	Let \(A\) be homogeneous, let \(\bw\) be generic, let \(\bv\) be a fake exponent of \(H_A(\bbeta)\) in the direction \(\bw\), and let \(\cN\) be an ordered negative support family for \(\bv\).
	By \cite[Theorem~5.5]{OS25}, each series obtained by ambient perturbation is a formal series solution of \(H_A(\bbeta)\) in the direction \(\bw\), and the coefficients of \(\bx^{\bv}\) in these series are exactly the elements of the full coefficient space at \(\bv\) for \(\cN\).
\end{fact}

Fact~\ref{fact:ambient-realization} states that the ambient coefficient space at \(\bv\) for \(\cN\) equals the full coefficient space at \(\bv\) for \(\cN\).

\begin{fact}
	\label{fact:colon-criterion}
	Let \(A\), \(\bw\), \(\bv\), and \(\cN\) be as in Fact~\ref{fact:ambient-realization}, and fix a Gale dual \(B\) of \(A\) and the resulting map \(\Phi\).
	By \cite[Proposition~6.2]{OS25}, the intrinsic coefficient space at \(\bv\) for \(\cN\) is contained in the ambient coefficient space, and the two are equal if and only if \(\Phi\bigl(Q_{\cN}(\bt):e\bigr) = \Pi_{\cN}:m_{\bv,\cN}\).
\end{fact}
Define the intrinsic-perturbation obstruction module by
\begin{equation}
	\fD_{\cN}(e)
	=
	\frac{(U+P_{\cN}(\bt)):e}
	{(UM_{\cN}+P_{\cN}(\bt)):e}.
	\label{eq:defect-module}
\end{equation}
Define also
\begin{equation}
	\fT_{\cN}
	=
	\frac{(U\cap M_{\cN})+P_{\cN}(\bt)}
	{UM_{\cN}+P_{\cN}(\bt)}.
	\label{eq:boundary-obstruction}
\end{equation}

\begin{theorem}
	\label{thm:lattice-obstruction}
	For a fixed Gale dual \(B\) and the resulting map \(\Phi\), there are natural \(\Shat\)-module isomorphisms
	\begin{align}
		\fD_{\cN}(e)
		&\simeq
		\frac{\Pi_{\cN}:m_{\bv,\cN}}
		{\Phi(Q_{\cN}(\bt):e)}
		\label{eq:defect-image}\\
		&\simeq
		\frac{
		e\Rhat\cap
		\bigl((U\cap M_{\cN})+P_{\cN}(\bt)\bigr)}
		{e\Rhat\cap
		\bigl(UM_{\cN}+P_{\cN}(\bt)\bigr)}.
		\label{eq:defect-intersection}
	\end{align}
	If \(\cN\) is ordered, the intrinsic coefficient space equals the ambient coefficient space if and only if \(\fD_{\cN}(e)=0\).

	Moreover, \eqref{eq:defect-intersection} identifies \(\fD_{\cN}(e)\) with the submodule of \(\fT_{\cN}\) consisting of the classes represented by multiples of \(e\).
	The module \(\fT_{\cN}\) is a quotient of
	\[
		\Tor^{\Rhat}_1
		\bigl(\Rhat/U,\Rhat/M_{\cN}\bigr)
		\simeq
		\frac{U\cap M_{\cN}}{UM_{\cN}}.
	\]
	In particular,
	\[
		U\cap M_{\cN}
		\subseteq UM_{\cN}+P_{\cN}(\bt)
		\quad\Longrightarrow\quad
		\fD_{\cN}(e)=0.
	\]
\end{theorem}

\begin{proof}
	Every generator of \(P_{\cN}(\bt)\) is divisible by a generator of \(M_{\cN}\), so \(P_{\cN}(\bt)\subseteq M_{\cN}\); also \(e\in M_{\cN}\), whence \(eU\subseteq UM_{\cN}\subseteq Q_{\cN}(\bt)\) and \(U\subseteq Q_{\cN}(\bt):e\).
	For \(f\in\Rhat\), the surjectivity of \(\Phi\) and the equality \(\ker\Phi=U\) make the following three conditions equivalent: \(f\in\Phi^{-1}(\Pi_{\cN}:m_{\bv,\cN})\), \(ef\in U+P_{\cN}(\bt)\), and \(f\in(U+P_{\cN}(\bt)):e\).
	Taking the quotient by \(Q_{\cN}(\bt):e\), which already contains \(U\), then proves \eqref{eq:defect-image}.
	If \(\cN\) is ordered, Fact~\ref{fact:colon-criterion} states that the intrinsic coefficient space equals the ambient coefficient space if and only if the two ideals in \eqref{eq:defect-image} are equal.
	Multiplication by \(e\) maps the quotient in \eqref{eq:defect-module} injectively onto \((e\Rhat\cap(U+P_{\cN}(\bt)))/(e\Rhat\cap Q_{\cN}(\bt))\), and since \(e\Rhat\subseteq M_{\cN}\) and \(P_{\cN}(\bt)\subseteq M_{\cN}\), the modular law gives \(M_{\cN}\cap(U+P_{\cN}(\bt))=(U\cap M_{\cN})+P_{\cN}(\bt)\), which proves \eqref{eq:defect-intersection} and the assertion concerning \(\fT_{\cN}\).
	Finally, tensoring \(0\to M_{\cN}\to\Rhat\to \Rhat/M_{\cN}\to0\) with \(\Rhat/U\) gives the description of \(\Tor_1\) in the statement; see \cite[Section~3.2]{Wei94}.
\end{proof}

\begin{remark}
	The word ``natural'' in Theorem~\ref{thm:lattice-obstruction} refers to the fixed Gale dual \(B\) and the resulting map \(\Phi\).
	If \(B'=BT\) with \(T\in\operatorname{GL}_r(\Z)\), then the substitution \(f(\bs)\mapsto f(T\bs)\) identifies the quotient \(J_{\cN}^{\intr}/J_{\cN}^{\amb}\) formed from \(B\) with the one formed from \(B'\).
	The same substitution identifies the completions of these two quotients and the resulting two presentations of \(\fD_{\cN}(e)\).
	Thus the obstruction module is compatible with changes of \(B\).
\end{remark}

\begin{remark}
	No finite-length hypothesis is imposed on \(\fD_{\cN}(e)\).
	When \(\fD_{\cN}(e)\) has finite length, we write \(\dim_{\C}\fD_{\cN}(e)\) for its dimension.
	For an ordered negative support family, Proposition~\ref{prop:defect-coefficient-duality} identifies this number with the dimension of the ambient coefficient space modulo the intrinsic coefficient space.
	Theorems~\ref{thm:lattice-counterexample}, \ref{thm:three-connected-obstruction}, and \ref{thm:fixed-rank-unbounded} concern this finite-length case.
\end{remark}

\subsection{Exact sequences of Koszul homology}

The exact sequences of this subsection describe the quotient of \(\Tor_1\) in Theorem~\ref{thm:lattice-obstruction} and the submodule of that quotient consisting of the classes represented by multiples of \(e\).
For the next theorem, abbreviate \(M=M_{\cN}\), \(P=P_{\cN}(\bt)\), and \(Q=UM+P\).

\begin{theorem}
	\label{thm:koszul-cokernels}
	There are natural right-exact sequences
	\begin{equation}
		\Tor^{\Rhat}_1
		(\Rhat/U,\Rhat/P)
		\longrightarrow
		\Tor^{\Rhat}_1
		(\Rhat/U,\Rhat/M)
		\longrightarrow
		\fT_{\cN}
		\longrightarrow0
		\label{eq:boundary-tor-cokernel}
	\end{equation}
	and
	\begin{equation}
		\Tor^{\Rhat}_1
		(\Rhat/\langle e\rangle,\Rhat/Q)
		\longrightarrow
		\Tor^{\Rhat}_1
		(\Rhat/\langle e\rangle,
		\Rhat/(U+P))
		\longrightarrow
		\fD_{\cN}(e)
		\longrightarrow0.
		\label{eq:distinguished-tor-cokernel}
	\end{equation}
	In \eqref{eq:boundary-tor-cokernel} and \eqref{eq:distinguished-tor-cokernel}, the first maps are induced by the quotient maps associated with \(P\subseteq M\) and with \(Q\subseteq U+P\), respectively.
	The second sequence is induced by the natural exact sequence
	\begin{align}
		\Tor^{\Rhat}_1
		(\Rhat/\langle e\rangle,\Rhat/Q)
		&\longrightarrow
		\Tor^{\Rhat}_1
		(\Rhat/\langle e\rangle,\Rhat/(U+P))
		\xrightarrow{\partial_e}
		\frac{U+P}{Q}\notag\\
		&\longrightarrow
		\frac{\Rhat}{Q+\langle e\rangle}
		\longrightarrow
		\frac{\Rhat}{U+P+\langle e\rangle}
		\longrightarrow0.
		\label{eq:distinguished-derived-exact-sequence}
	\end{align}
	The image of \(\partial_e\) is naturally isomorphic to \(\fD_{\cN}(e)\).
	Equivalently,
	\begin{equation}
		\fD_{\cN}(e)
		\simeq
		\ker\left(
		\frac{U+P}{Q}
		\longrightarrow
		\frac{\Rhat}{Q+\langle e\rangle}
		\right).
		\label{eq:distinguished-specialization-kernel}
	\end{equation}
	Under multiplication by \(e\), the last term in \eqref{eq:distinguished-tor-cokernel} is identified with the submodule of \(\fT_{\cN}\) consisting of the classes represented by multiples of \(e\).
\end{theorem}

\begin{proof}
	The entries \(\theta_1,\ldots,\theta_d\) of \(A\bt\) are linearly independent, so the Koszul complex on them resolves \(\Rhat/U\).
	Tensoring that complex with \(0\to M/P\to\Rhat/P\to\Rhat/M\to0\) gives the homology sequence
	\[
		\Tor^{\Rhat}_1(\Rhat/U,\Rhat/P) \longrightarrow \Tor^{\Rhat}_1(\Rhat/U,\Rhat/M) \longrightarrow M/(UM+P) \longrightarrow \Rhat/(U+P),
	\]
	whose last map has kernel \(M\cap(U+P)/(UM+P)=((U\cap M)+P)/(UM+P)=\fT_{\cN}\) by the modular law.
	This proves \eqref{eq:boundary-tor-cokernel}.

	Since \(e\) is a nonzero monomial, the one-element Koszul complex on \(e\) resolves \(\Rhat/\langle e\rangle\) and gives \(\Tor^{\Rhat}_1(\Rhat/\langle e\rangle,\Rhat/Q)\simeq(Q:e)/Q\) and \(\Tor^{\Rhat}_1(\Rhat/\langle e\rangle,\Rhat/(U+P))\simeq((U+P):e)/(U+P)\).
	The inclusions \(eU\subseteq UM\) and \(eP\subseteq P\) give \(U+P\subseteq Q:e\), so the image of the first map in \eqref{eq:distinguished-tor-cokernel} is \((Q:e)/(U+P)\) and its cokernel is \(\fD_{\cN}(e)\).
	The same inclusions give \(e(U+P)\subseteq Q\), so \(e\) annihilates \((U+P)/Q\), and applying \(\Rhat/\langle e\rangle\otimes_{\Rhat}^{\mathbf L}-\) to \(0\to(U+P)/Q\to\Rhat/Q\to\Rhat/(U+P)\to0\) yields \eqref{eq:distinguished-derived-exact-sequence}.
	At the chain level \(\partial_e\) is multiplication by \(e\), which sends the class of \(f\in(U+P):e\) to the class of \(ef\) in \(\fT_{\cN}\); this class is zero exactly when \(f\in Q:e\), so the image of \(\partial_e\) consists of the classes represented by multiples of \(e\), and the last assertion follows.
\end{proof}

\subsection{A finite presentation of the two colon ideals}
\label{subsec:finite-boundary-presentation}

We next describe the two colon ideals in Theorem~\ref{thm:lattice-obstruction} by a finite presentation.
This presentation gives a necessary and sufficient criterion for the vanishing of \(\fD_{\cN}(e)\) in terms of finite data.

Set \(R_0=\C[t_1,\ldots,t_n]\), \(S_0=\C[s_1,\ldots,s_r]\), and \(\fm=\langle s_1,\ldots,s_r\rangle\).
For \(1\leq i\leq d\), put \(\theta_i=\sum_{j=1}^n a_{ij}t_j\), and put
\[
	U_0=\langle\theta_1,\ldots,\theta_d\rangle,
	\qquad
	\Phi_0:R_0\longrightarrow S_0,
	\qquad
	t_j\longmapsto\ell_j=(B\bs)_j.
\]
Thus \(\ker\Phi_0=U_0\).
Let \(M_{\cN}^0\) and \(P_{\cN}^0(\bt)\) be the monomial ideals of \(R_0\) generated by the monomials in \eqref{eq:obstruction-M} and \eqref{eq:obstruction-P}, respectively, and write \(Q_{\cN}^0(\bt)=U_0M_{\cN}^0+P_{\cN}^0(\bt)\).
For a set \(Y\subseteq\{1,\ldots,n\}\), use the notation \(\ell^Y=\prod_{j\in Y}\ell_j\) and \(\ell^{\varnothing}=1\).

The two antichains used below are
\begin{align}
	\cG_{\cN}
	&=
	\min_{\subseteq}
	\{I\setminus K_{\cN}\mid I\in\cN\}
	=\{G_1,\ldots,G_g\},
	\label{eq:boundary-G-antichain}\\
	\cH_{\cN}
	&=
	\min_{\subseteq}
	\{(I\cup J)\setminus K_{\cN}
	\mid I\in\cN,\
	J\in\sS(\bv)\setminus\cN\}
	=\{H_1,\ldots,H_h\}.
	\label{eq:boundary-H-antichain}
\end{align}
The second antichain is empty when \(\sS(\bv)=\cN\).
Each \(H_a\) contains \(G_i\) for some index \(i\).
Choose one such index \(\iota(a)\).
Let
\[
	F_{\cN}
	=
	\bigoplus_{i=1}^g S_0(-|G_i|)\eta_i
\]
be the free graded \(S_0\)-module with basis \(\eta_1,\ldots,\eta_g\), and let \(Z_{\cN}\subseteq F_{\cN}\) be the submodule generated by
\begin{align}
	\tau_{ij}
	&=
	\ell^{G_j\setminus G_i}\eta_i
	-
	\ell^{G_i\setminus G_j}\eta_j
	&&(1\leq i<j\leq g),
	\label{eq:boundary-Taylor-relations}\\
	\rho_a
	&=
	\ell^{H_a\setminus G_{\iota(a)}}\eta_{\iota(a)}
	&&(1\leq a\leq h).
	\label{eq:boundary-generator-relations}
\end{align}
Put \(E=I_{\bzero}\setminus K_{\cN}\).
Since \(I_{\bzero}\in\cN\), some \(G_i\) is contained in \(E\).
For any such \(i\), define
\begin{equation}
	\epsilon_{\cN}
	=
	\ell^{E\setminus G_i}\overline{\eta_i}
	\in F_{\cN}/Z_{\cN}.
	\label{eq:distinguished-boundary-class}
\end{equation}
Finally, define two homogeneous ideals of \(S_0\) by
\begin{equation}
	J_{\cN}^{\amb}=\Ann_{S_0}(\epsilon_{\cN}), \qquad J_{\cN}^{\intr}=\langle\ell^{H_1},\ldots,\ell^{H_h}\rangle:\ell^E.
	\label{eq:boundary-comparison-ideals}
\end{equation}
For every \(j\in E\), there are \(I,I'\in\cN\) with \(j\in I\) and \(j\notin I'\), so the \(j\)-th row of \(B\) and the linear form \(\ell_j\) are nonzero.
Thus \(\ell^E\ne0\).
If \(h=0\), then \(J_{\cN}^{\intr}\) is zero.
For the distinguished collection, if \(\rank L\geq1\), then we have \(h\geq1\), because \(L\) contains some \(\bu\) with \(\bw\cdot\bu<0\), and such a \(\bu\) gives \(I_{\bu}\in\sS(\bv)\setminus\cN_{\bv}\).

For the completed obstruction module, we write \(u\) for the variable of the Hilbert series and use the notation
\begin{equation}
	\Hilb
	\bigl(\fD_{\cN}(e);u\bigr)
	=
	\sum_{q\geq0}
	\dim_{\C}
	\left(
	J_{\cN}^{\intr}/J_{\cN}^{\amb}
	\right)_q u^q.
	\label{eq:boundary-completed-Hilbert-definition}
\end{equation}
The definition in \eqref{eq:boundary-completed-Hilbert-definition} uses the graded quotient \(J_{\cN}^{\intr}/J_{\cN}^{\amb}\); we do not regard the completed module as a direct-sum graded module.

Let \(\Gamma_{\cN}\) be the matrix whose columns are the vectors in \eqref{eq:boundary-Taylor-relations} and \eqref{eq:boundary-generator-relations}, followed, as the last column, by a lift of \(\epsilon_{\cN}\) to \(F_{\cN}\).
Grade the source so that \(\Gamma_{\cN}\) becomes a homogeneous map \(\bigoplus_jS_0(-d_j)\to F_{\cN}\), where \(d_j\) is the degree of the \(j\)-th column.
For a vector \(z\), write \(z_{\mathrm{last}}\) for its last coordinate.

\begin{theorem}
	\label{thm:finite-boundary-classification}
	Up to a graded relabeling of the basis, the submodule \(Z_{\cN}\), the class \(\epsilon_{\cN}\), and the ideals \(J_{\cN}^{\amb}\) and \(J_{\cN}^{\intr}\) are independent of the orderings and choices made above.
	Moreover,
	\begin{equation}
		J_{\cN}^{\amb}
		=
		\Phi_0\bigl(
		Q_{\cN}^0(\bt):e
		\bigr)
		\subseteq
		J_{\cN}^{\intr},
		\label{eq:boundary-annihilator-colon}
	\end{equation}
	and there is an isomorphism
	\begin{equation}
		\fD_{\cN}(e)
		\simeq
		\Shat\otimes_{S_{0,\fm}}
		\left(
		\frac{J_{\cN}^{\intr}}{J_{\cN}^{\amb}}
		\right)_{\fm}
		\simeq
		\Shat\otimes_{S_0}
		\frac{J_{\cN}^{\intr}}{J_{\cN}^{\amb}}.
		\label{eq:boundary-classification-isomorphism}
	\end{equation}
	In particular,
	\begin{equation}
		\fD_{\cN}(e)=0
		\quad\Longleftrightarrow\quad
		J_{\cN}^{\amb}=J_{\cN}^{\intr}.
		\label{eq:boundary-classification-equality}
	\end{equation}
	Furthermore,
	\begin{equation}
		\Hilb
		\bigl(\fD_{\cN}(e);u\bigr)
		=
		\Hilb(S_0/J_{\cN}^{\amb};u)
		-
		\Hilb(S_0/J_{\cN}^{\intr};u).
		\label{eq:boundary-classification-Hilbert}
	\end{equation}
	For the distinguished collection at a fake exponent, both quotient rings on the right are Artinian, and hence
	\begin{equation}
		\dim_{\C}\fD_{\cN_{\bv}}(e)
		=
		\dim_{\C}(S_0/J_{\cN_{\bv}}^{\amb})
		-
		\dim_{\C}(S_0/J_{\cN_{\bv}}^{\intr}).
		\label{eq:boundary-classification-length}
	\end{equation}

	If \(\cZ\) is any generating set of \(\ker\Gamma_{\cN}\), then
	\begin{equation}
		J_{\cN}^{\amb}
		=
		\left\langle
		z_{\mathrm{last}}\ \middle|\ z\in\cZ
		\right\rangle.
		\label{eq:boundary-matrix-kernel}
	\end{equation}
	The extensions of \(J_{\cN}^{\amb}\) and \(J_{\cN}^{\intr}\) to \(\Shat\) are \(\Phi(Q_{\cN}(\bt):e)\) and \(\Pi_{\cN}:m_{\bv,\cN}\), respectively.
	Thus the obstruction module is determined by the antichains \(\cG_{\cN}\) and \(\cH_{\cN}\), the monomial \(e\), and the linear forms \(\ell_j\), and hence so are its vanishing and its Hilbert series.
\end{theorem}

\begin{proof}
	We first work over \(R_0\).
	Set
	\[
		\widetilde F_{\cN}
		=
		\bigoplus_{i=1}^gR_0(-|G_i|)\eta_i.
	\]
	Map the free graded \(R_0\)-module \(\widetilde F_{\cN}\) onto \(M_{\cN}^0\) by \(\eta_i\longmapsto\bt^{G_i}\).
	The kernel of this map is generated by the pairwise Taylor relations \(\bt^{G_j\setminus G_i}\eta_i - \bt^{G_i\setminus G_j}\eta_j\).
	This description of the kernel is the first-syzygy part of the Taylor resolution \cite[Chapter~4]{MS05}.
	For completeness, this description of the kernel also follows from the \(\Z^n\)-grading.
	In a fixed multidegree \(\balpha\), every nonzero term of a syzygy is a scalar multiple of \(\bt^{ \balpha-\bone_{G_i}}\eta_i\), and the sum of those scalars is zero; here \(\bone_Y\) denotes the indicator vector of a set \(Y\).
	Fixing one index \(i\) whose term is nonzero, we express the vanishing of that sum of scalars as a sum of pairwise differences.
	Multiplying the corresponding Taylor relations by \(\bt^{ \balpha-\bone_{G_i\cup G_j}}\) gives the original homogeneous syzygy.

	The ideal \(P_{\cN}^0(\bt)\) is generated by the monomials \(\bt^{H_a}\).
	Since \(G_{\iota(a)}\subseteq H_a\), the element \(\bt^{H_a\setminus G_{\iota(a)}}\eta_{\iota(a)}\) is a lift of \(\bt^{H_a}\).
	Consequently, adjoining these \(h\) relations to the Taylor relations yields a presentation of \(M_{\cN}^0/P_{\cN}^0(\bt)\).
	The lifts corresponding to two choices of \(\iota(a)\) differ by a Taylor relation.
	Hence the resulting relation module does not depend on the choice of \(\iota(a)\).

	Tensor this presentation with \(S_0=R_0/U_0\).
	Under the identification \(S_0\otimes_{R_0}\widetilde F_{\cN}=F_{\cN}\), the images of the Taylor relations and of the \(h\) additional relations generate \(Z_{\cN}\).
	Right exactness of the tensor product gives
	\begin{equation}
		\frac{F_{\cN}}{Z_{\cN}}
		\simeq
		S_0\otimes_{R_0}
		\frac{M_{\cN}^0}
		{P_{\cN}^0(\bt)}
		\simeq
		\frac{M_{\cN}^0}
		{U_0M_{\cN}^0+
		P_{\cN}^0(\bt)}.
		\label{eq:boundary-presented-module}
	\end{equation}
	Under the isomorphisms in \eqref{eq:boundary-presented-module}, the class in \eqref{eq:distinguished-boundary-class} is the class of \(e=\bt^E\).
	In particular, the class \(\epsilon_{\cN}\) is independent of the chosen \(G_i\subseteq E\).
	The annihilator of this class is therefore \(\left\{ \Phi_0(f)\ \middle|\ fe\in U_0M_{\cN}^0+ P_{\cN}^0(\bt) \right\} = \Phi_0\bigl( Q_{\cN}^0(\bt):e \bigr)\), which proves the equality in \eqref{eq:boundary-annihilator-colon}.

	On the other hand, \(\Phi_0(P_{\cN}^0(\bt))=\langle\ell^{H_1},\ldots,\ell^{H_h}\rangle\) and \(\Phi_0(e)=\ell^E\).
	Thus \(J_{\cN}^{\intr}\) is the counterpart of \(\Pi_{\cN}:m_{\bv,\cN}\) in \(S_0\).
	If \(\Phi_0(f)\in J_{\cN}^{\amb}\), applying \(\Phi_0\) to \(fe\in U_0M_{\cN}^0+ P_{\cN}^0(\bt)\) shows that \(\Phi_0(f)\ell^E\) belongs to \(\Phi_0(P_{\cN}^0(\bt))\).
	Hence \(J_{\cN}^{\amb}\subseteq J_{\cN}^{\intr}\).

	The completion maps \(R_{0,\langle t_1,\ldots,t_n\rangle}\to\Rhat\) and \(S_{0,\fm}\to\Shat\) are faithfully flat, and forming a colon ideal commutes with flat base change, because \((\fa:y)/\fa\) is the kernel of multiplication by \(y\) on the quotient by \(\fa\).
	Since \(U_0\subseteq Q_{\cN}^0(\bt):e\), it follows that \(J_{\cN}^{\amb}\Shat=\Phi\bigl(Q_{\cN}(\bt):e\bigr)\) and, by the same argument over \(S_0\), that \(J_{\cN}^{\intr}\Shat = \Pi_{\cN}:m_{\bv,\cN}\), so Theorem~\ref{thm:lattice-obstruction} gives \eqref{eq:boundary-classification-isomorphism}.
	The quotient \(J_{\cN}^{\intr}/J_{\cN}^{\amb}\) is a finitely generated graded \(S_0\)-module, and a nonzero homogeneous element of such a module survives localization at \(\fm\), so this quotient vanishes if and only if its localization does; faithful flatness of \(S_{0,\fm}\to\Shat\) then gives \eqref{eq:boundary-classification-equality}.
	The \(\fm\)-adic completion of a finitely generated graded module is the product of its homogeneous components, which are the components used in \eqref{eq:boundary-completed-Hilbert-definition}, and the exact sequence \(0\to J_{\cN}^{\intr}/J_{\cN}^{\amb}\to S_0/J_{\cN}^{\amb}\to S_0/J_{\cN}^{\intr}\to0\) gives \eqref{eq:boundary-classification-Hilbert}.

	It remains to bound \(\dim_{\C}(S_0/J_{\cN}^{\amb})\) for the distinguished collection at a fake exponent.
	Since \(\bw\cdot\bg^{(i)}>0\), the point \(-\bg^{(i)}\) has negative weight and lies in the fiber of \(I_{-\bg^{(i)}}\), so \(I_{-\bg^{(i)}}\in\sS(\bv)\setminus\cN_{\bv}\).
	With \(I=I_{\bzero}\) and \(J=I_{-\bg^{(i)}}\) in \eqref{eq:obstruction-P}, and with \(K_{\cN}\subseteq I_{\bzero}\) and \(G^{(i)}(\bv)\cap I_{\bzero}=\varnothing\), the resulting generator is \(e\,\bt^{G^{(i)}(\bv)}\); hence \(\bt^{G^{(i)}(\bv)}\in Q_{\cN}^0(\bt):e\), and \eqref{eq:boundary-annihilator-colon} gives \(P_B(\bv)\subseteq J_{\cN}^{\amb}\).

	We show that \(P_B(\bv)\) is primary to \(\fm\).
	Let \(\bs_0\in\C^r\) be a common zero of the generators of \(P_B(\bv)\), put \(\bu=B\bs_0\), and fix \(i\).
	Since \(\prod_{j\in G^{(i)}(\bv)}\ell_j(\bs_0)=0\), we may choose \(j\in G^{(i)}(\bv)\) with \(u_j=\ell_j(\bs_0)=0\); then \(v_j\) is not a negative integer but \(v_j-\bg^{(i)}_j\) is, so \(v_j\in\{0,1,\ldots,(\bg^{(i)}_+)_j-1\}\) and the distraction of \(\bpartial^{\bg^{(i)}_+}\) vanishes along the whole line \(\bv+\C\bu\).
	Since \(AB=0\), every point of \(\bv+\C\bu\) has \(A\)-image \(\bbeta\), and the description of the fake indicial ideal by distractions in \cite[Section~3.2]{SST00} shows that every such point is a fake exponent.
	There are only finitely many fake exponents for \(\bbeta\) and \(\bw\).
	For a standard pair \((\ba,\sigma)\) of \(\inw(I_A)\), the columns of \(A\) indexed by \(\sigma\) are linearly independent.
	Otherwise, a nonzero integral relation among these columns would give a lattice binomial supported in \(\sigma\).
	The initial monomial of this binomial, multiplied by \(\bpartial^{\ba}\), would belong to \(\inw(I_A)\).
	This membership contradicts the defining property that every monomial obtained by multiplying \(\bpartial^{\ba}\) by a monomial supported in \(\sigma\) is standard for \(\inw(I_A)\).
	The conditions \(A\bv=\bbeta\) and \(v_j=a_j\) for \(j\notin\sigma\) therefore determine \(\bv\), and there are finitely many standard pairs.
	Hence \(\bu=\bzero\), and \(\bs_0=\bzero\) because the columns of \(B\) are linearly independent.
	Hilbert's Nullstellensatz gives \(\sqrt{P_B(\bv)}=\fm\), so \(\dim_{\C}(S_0/J_{\cN}^{\amb}) \leq \dim_{\C}(S_0/P_B(\bv)) <\infty\) and \(S_0/J_{\cN}^{\amb}\) is Artinian; when the quotient \(S_0/J_{\cN}^{\amb}\) is nonzero, \(J_{\cN}^{\amb}\) is primary to \(\fm\).
	Since \(J_{\cN}^{\amb}\subseteq J_{\cN}^{\intr}\), the quotient \(S_0/J_{\cN}^{\intr}\) is Artinian as well, and it is zero if \(J_{\cN}^{\intr}=S_0\), so the Hilbert-series identity gives \eqref{eq:boundary-classification-length}.
	Finally, an element \(f\in S_0\) annihilates \(\epsilon_{\cN}\) exactly when \(\Gamma_{\cN}\tp{(\bq,-f)}=0\) for some vector \(\bq\) of relation coefficients, so the image of \(\ker\Gamma_{\cN}\) under the last-coordinate projection is the annihilator ideal, which proves \eqref{eq:boundary-matrix-kernel}.
\end{proof}

For a homogeneous subspace \(W\subseteq S_{0,i}\), let \(W^\perp\subseteq\C[\bpartial_{\bs}]_i\) denote its orthogonal complement under \(\langle\cdot,\cdot\rangle_{\bs}\).
Put \(V_{\cN,i}^{\amb} = (J_{\cN}^{\amb})_i^\perp\) and \(V_{\cN,i}^{\intr} = (J_{\cN}^{\intr})_i^\perp\).

\begin{proposition}
	\label{prop:defect-coefficient-duality}
	For every \(i\geq0\), there is an inclusion \(V_{\cN,i}^{\intr} \subseteq V_{\cN,i}^{\amb}\), and \(\langle\cdot,\cdot\rangle_{\bs}\) induces a perfect pairing
	\begin{equation}
		\left(
		\frac{J_{\cN}^{\intr}}
		{J_{\cN}^{\amb}}
		\right)_i
		\times
		\frac{V_{\cN,i}^{\amb}}
		{V_{\cN,i}^{\intr}}
		\longrightarrow\C.
		\label{eq:defect-coefficient-perfect-pairing}
	\end{equation}
	In particular,
	\[
		\left(
		\frac{J_{\cN}^{\intr}}
		{J_{\cN}^{\amb}}
		\right)_i
		\simeq
		\left(
		\frac{V_{\cN,i}^{\amb}}
		{V_{\cN,i}^{\intr}}
		\right)^\vee.
	\]
	If \(\cN\) is ordered, Theorem~5.5 and Proposition~6.2 of \cite{OS25} identify \(V_{\cN,i}^{\amb}\) and \(V_{\cN,i}^{\intr}\) with the degree-\(i\) ambient and intrinsic coefficient spaces at \(\bv\), respectively.
\end{proposition}

\begin{proof}
	The inclusion \(J_{\cN}^{\amb}\subseteq J_{\cN}^{\intr}\) gives the reverse inclusion of the orthogonal complements.
	The restriction of \(\langle\cdot,\cdot\rangle_{\bs}\) to \(S_{0,i}\times\C[\bpartial_{\bs}]_i\) is perfect.
	Finite-dimensional linear duality therefore identifies the dual of \((J_{\cN}^{\intr}/J_{\cN}^{\amb})_i\) with \(V_{\cN,i}^{\amb}/V_{\cN,i}^{\intr}\), and this identification proves \eqref{eq:defect-coefficient-perfect-pairing}.
	The last assertion follows from the two cited results and the description of \(J_{\cN}^{\intr}/J_{\cN}^{\amb}\) in Theorem~\ref{thm:finite-boundary-classification}.
\end{proof}

The second right-exact sequence \eqref{eq:distinguished-tor-cokernel} also gives the following description of the quotient of the ambient coefficient space by the intrinsic coefficient space.
This description does not involve the completion.

\begin{corollary}
	\label{cor:tor-coefficient-exact-sequence}
	Put
	\begin{equation}
		\begin{aligned}
			\fC_{\cN}
			=
			\coker\Bigl(&
			\Tor^{R_0}_1
			(R_0/\langle e\rangle,R_0/Q_{\cN}^0(\bt))\\
			&\longrightarrow
			\Tor^{R_0}_1
			(R_0/\langle e\rangle,
			R_0/(U_0+P_{\cN}^0(\bt)))
			\Bigr).
		\end{aligned}
		\label{eq:polynomial-tor-cokernel}
	\end{equation}
	There is a natural graded isomorphism
	\begin{equation}
		\fC_{\cN}(|E|)
		\simeq
		\frac{J_{\cN}^{\intr}}
		{J_{\cN}^{\amb}},
		\label{eq:polynomial-tor-obstruction-isomorphism}
	\end{equation}
	where the shift by \(|E|\) cancels the Koszul shift by \(\deg e=|E|\).
	For every \(i\geq0\), there is an exact sequence
	\begin{equation}
		0\longrightarrow
		V_{\cN,i}^{\intr}
		\longrightarrow
		V_{\cN,i}^{\amb}
		\longrightarrow
		\bigl(\fC_{\cN}(|E|)_i\bigr)^\vee
		\longrightarrow0.
		\label{eq:tor-coefficient-exact-sequence}
	\end{equation}
	If \(\cN\) is ordered, the first two terms are the intrinsic and ambient coefficient spaces at \(\bv\), respectively.
\end{corollary}

\begin{proof}
	Applying the one-element Koszul calculation in Theorem~\ref{thm:koszul-cokernels} over \(R_0\) gives \(\fC_{\cN}(|E|)\simeq((U_0+P_{\cN}^0(\bt)):e)/(Q_{\cN}^0(\bt):e)\).
	For \(f\in R_0\), the membership \(f\in(U_0+P_{\cN}^0(\bt)):e\) is equivalent to \(\Phi_0(f)\Phi_0(e)\in\Phi_0(P_{\cN}^0(\bt))\), hence to \(\Phi_0(f)\in J_{\cN}^{\intr}\).
	Equation \eqref{eq:boundary-annihilator-colon} identifies the image of the denominator with \(J_{\cN}^{\amb}\).
	Since \(U_0\subseteq Q_{\cN}^0(\bt):e\), taking the quotient proves \eqref{eq:polynomial-tor-obstruction-isomorphism}, and Proposition~\ref{prop:defect-coefficient-duality} gives \eqref{eq:tor-coefficient-exact-sequence} and the last assertion.
\end{proof}

\subsection{Two elements in the antichain \texorpdfstring{\(\cG_{\cN}\)}{G(N)}}

In this subsection, we compute \(J_{\cN}^{\amb}\) and \(J_{\cN}^{\intr}\) when \(\cG_{\cN}\) has two elements.
Throughout this subsection, we assume that \(\cG_{\cN}=\{G_1,G_2\}\), and we label these sets so that \(G_1\subseteq E\).
Define \(\alpha\), \(\beta\), and \(\chi\) by \(\alpha=\ell^{G_1\setminus G_2}\), \(\beta=\ell^{G_2\setminus G_1}\), and \(\chi=\ell^{E\setminus G_1}\).
For \(i\in\{1,2\}\), define the homogeneous ideal \(\fa_i\) of \(S_0\) by
\begin{equation}
	\fa_i
	=
	\left\langle
	\ell^{H_a\setminus G_i}
	\ \middle|\
	1\leq a\leq h,\ \iota(a)=i
	\right\rangle,
	\label{eq:two-support-multiplier-ideals}
\end{equation}
where the ideal generated by the empty set is zero.

The following theorem gives \(J_{\cN}^{\amb}\) and \(J_{\cN}^{\intr}\).

\begin{theorem}
	\label{thm:two-minimal-support-formula}
	We have
	\begin{equation}
		J_{\cN}^{\amb}=\bigl(\fa_1+\beta(\fa_2:\alpha)\bigr):\chi, \qquad J_{\cN}^{\intr}=\bigl(\fa_1+(\beta\fa_2:\alpha)\bigr):\chi.
		\label{eq:two-support-formulas}
	\end{equation}
\end{theorem}

\begin{proof}
	The Taylor relation between \(\eta_1\) and \(\eta_2\) is \(\beta\eta_1-\alpha\eta_2\), so the presentation in Theorem~\ref{thm:finite-boundary-classification} gives
	\begin{equation}
		\frac{F_{\cN}}{Z_{\cN}}
		\simeq
		\frac{S_0\eta_1\oplus S_0\eta_2}
		{\langle
			\beta\eta_1-\alpha\eta_2,
			\fa_1\eta_1,
			\fa_2\eta_2
			\rangle}.
		\label{eq:two-support-module-presentation}
	\end{equation}
	An element \(f\in S_0\) annihilates \(\overline{\eta_1}\) if and only if there are \(c\in S_0\), \(a\in\fa_1\), and \(b\in\fa_2\) with \(f=c\beta+a\) and \(c\alpha=b\), so \(\Ann_{S_0}(\overline{\eta_1})=\fa_1+\beta(\fa_2:\alpha)\).
	Since \(\epsilon_{\cN}=\chi\overline{\eta_1}\), taking the colon by \(\chi\) proves the formula for \(J_{\cN}^{\amb}\) in \eqref{eq:two-support-formulas}.

	For every \(j\in G_1\cup G_2\), there are \(I,I'\in\cN\) with \(j\in I\) and \(j\notin I'\), so the \(j\)-th row of \(B\) and the linear form \(\ell_j\) are nonzero.
	Define \(\kappa\) by \(\kappa=\ell^{G_1\cap G_2}\).
	The generators lifted from \(G_1\) have images under \(\Phi_0\) that generate \(\kappa\alpha\fa_1\), and those lifted from \(G_2\) have images that generate \(\kappa\beta\fa_2\), so \(\langle\ell^{H_1},\ldots,\ell^{H_h}\rangle=\kappa(\alpha\fa_1+\beta\fa_2)\) and \(\ell^E=\kappa\alpha\chi\).
	Canceling \(\kappa\) in the domain \(S_0\) gives \(J_{\cN}^{\intr}=(\alpha\fa_1+\beta\fa_2):\alpha\chi\).
	If \(f\alpha=a\alpha+\beta b\) with \(a\in\fa_1\) and \(b\in\fa_2\), then \(f-a\in(\beta\fa_2:\alpha)\), and the reverse inclusion follows from the same equality, so \((\alpha\fa_1+\beta\fa_2):\alpha=\fa_1+(\beta\fa_2:\alpha)\).
	Taking the colon by \(\chi\) proves the formula for \(J_{\cN}^{\intr}\) in \eqref{eq:two-support-formulas}.
\end{proof}

When \(E=G_1\) and \(\fa_2\) is principal, the quotient \(J_{\cN}^{\intr}/J_{\cN}^{\amb}\) of the ideals in Theorem~\ref{thm:two-minimal-support-formula} is cyclic.

\begin{corollary}
	\label{cor:two-support-principal-formula}
	Suppose that \(E=G_1\) and that there is a nonzero homogeneous polynomial \(\gamma\) such that \(\fa_2=\langle\gamma\rangle\).
	Choose representatives of the greatest common divisors, which are defined up to nonzero constants, and define \(g_0\), \(h_0\), and \(\omega\) by \(g_0=\gcd(\alpha,\gamma)\), \(h_0=\gcd(\alpha/g_0,\beta)\), and \(\omega=(\beta/h_0)(\gamma/g_0)\).
	Then
	\begin{equation}
		J_{\cN}^{\amb}=\fa_1+\langle h_0\omega\rangle, \qquad J_{\cN}^{\intr}=\fa_1+\langle\omega\rangle,
		\label{eq:two-support-principal-formulas}
	\end{equation}
	and multiplication by \(\omega\) gives a graded isomorphism
	\begin{equation}
		\frac{S_0}{\langle h_0\rangle+(\fa_1:\omega)}
		(-\deg\omega)
		\longrightarrow
		\frac{J_{\cN}^{\intr}}
		{J_{\cN}^{\amb}}.
		\label{eq:two-support-cyclic-model}
	\end{equation}
	Suppose further that there is a nonzero homogeneous polynomial \(\delta\) such that \(\fa_1=\langle\delta\rangle\), and define \(\delta_{\omega}\), \(d_h\), \(d_{\delta}\), and \(d_{h,\delta}\) by \(\delta_{\omega}=\delta/\gcd(\delta,\omega)\), \(d_h=\deg h_0\), \(d_{\delta}=\deg\delta_{\omega}\), and \(d_{h,\delta}=\deg\gcd(h_0,\delta_{\omega})\).
	Then
	\begin{equation}
		\frac{J_{\cN}^{\intr}}
		{J_{\cN}^{\amb}}
		\simeq
		\frac{S_0}{\langle h_0,\delta_{\omega}\rangle}
		(-\deg\omega),
		\label{eq:two-support-gcd-model}
	\end{equation}
	and
	\begin{equation}
		\Hilb\bigl(\fD_{\cN}(e);u\bigr)
		=
		u^{\deg\omega}
		\frac{1-u^{d_h}-u^{d_{\delta}}+u^{d_h+d_{\delta}-d_{h,\delta}}}{(1-u)^r}.
		\label{eq:two-support-gcd-Hilbert}
	\end{equation}
	The obstruction vanishes if and only if \(h_0\) or \(\delta_{\omega}\) is a nonzero constant.
\end{corollary}

\begin{proof}
	Write \(\alpha=g_0\alpha'\) and \(\gamma=g_0\gamma'\), where \(\alpha'\) and \(\gamma'\) are relatively prime, so that \(\gcd(\alpha,\beta\gamma)=g_0\gcd(\alpha',\beta)=g_0h_0\).
	Theorem~\ref{thm:two-minimal-support-formula} and Euclid's lemma give \(\beta(\langle\gamma\rangle:\alpha)=\langle\beta\gamma/g_0\rangle=\langle h_0\omega\rangle\) and \((\beta\langle\gamma\rangle:\alpha)=\langle\beta\gamma/\gcd(\alpha,\beta\gamma)\rangle=\langle\omega\rangle\), which proves \eqref{eq:two-support-principal-formulas}.
	The quotient of these two ideals is generated by the class of \(\omega\), and an element \(f\in S_0\) annihilates this class if and only if \(f\omega\in\fa_1+\langle h_0\omega\rangle\), which is equivalent to \(f\in(\fa_1:\omega)+\langle h_0\rangle\); this proves \eqref{eq:two-support-cyclic-model}.

	If \(\fa_1=\langle\delta\rangle\), then Euclid's lemma gives \((\fa_1:\omega)=\langle\delta_{\omega}\rangle\), which proves \eqref{eq:two-support-gcd-model}.
	The first syzygy module of \(\langle h_0,\delta_{\omega}\rangle\) is generated by the syzygy obtained by dividing the two generators by their greatest common divisor, so \(S_0/\langle h_0,\delta_{\omega}\rangle\) has the graded free resolution
	\[
		0
		\longrightarrow
		S_0(-d_h-d_{\delta}+d_{h,\delta})
		\longrightarrow
		S_0(-d_h)\oplus S_0(-d_{\delta})
		\longrightarrow
		S_0
		\longrightarrow
		\frac{S_0}{\langle h_0,\delta_{\omega}\rangle}
		\longrightarrow0.
	\]
	Taking the Hilbert series and applying the shift in \eqref{eq:two-support-gcd-model} proves \eqref{eq:two-support-gcd-Hilbert}.
	The polynomials \(h_0\) and \(\delta_{\omega}\) generate the unit ideal if and only if one of them is a nonzero constant, which proves the last assertion.
\end{proof}

Corollary~\ref{cor:two-support-principal-formula} gives the following condition on the rows of \(B\).

\begin{corollary}
	\label{cor:two-support-proportional-rows}
	Suppose that \(E=G_1\), that there is a nonzero homogeneous polynomial \(\gamma\) such that \(\fa_2=\langle\gamma\rangle\), and that \(\fD_{\cN}(e)\ne0\).
	Then there are indices \(i\in G_1\setminus G_2\) and \(j\in G_2\setminus G_1\) such that the \(i\)-th and \(j\)-th rows of \(B\) are nonzero and proportional.
\end{corollary}

\begin{proof}
	The isomorphism \eqref{eq:two-support-cyclic-model} shows that \(h_0\) is not a nonzero constant.
	Every irreducible factor of \(h_0=\gcd(\alpha/g_0,\beta)\) is a linear form that divides one factor of \(\alpha\) and one factor of \(\beta\).
	The corresponding indices lie in the two asserted sets, and the corresponding rows of \(B\) are proportional.
\end{proof}

\subsection{A Cohen--Macaulay vanishing criterion}

In this subsection, we give a condition under which the module \(\fT_{\cN}\) vanishes.
Retain \(R_0\), \(U_0\), and \(M_{\cN}^0\) from Section~\ref{subsec:finite-boundary-presentation}.
If \(M_{\cN}^0=R_0\), then \(M_{\cN}=\Rhat\) and every conclusion of Theorem~\ref{thm:CM-regular-sequence} holds directly; henceforth assume that \(M_{\cN}^0\) is proper.
Since the monomial generators of \(M_{\cN}^0\) are squarefree, there is a simplicial complex \(\Delta_{\cN}\) whose Stanley--Reisner ideal \(I_{\Delta_{\cN}}\) equals \(M_{\cN}^0\).
For a face \(F\) of \(\Delta_{\cN}\), denote by \(A_F\) the submatrix of \(A\) whose columns are indexed by \(F\).
For the standard facts about Cohen--Macaulay Stanley--Reisner rings, systems of parameters, and regular sequences, we refer to \cite[Chapters~1--2]{BH98} and \cite[Chapter~13]{MS05}.

\begin{theorem}
	\label{thm:CM-regular-sequence}
	Let \(\bw\) be generic, let \(\bv\) be a fake exponent, and let \(\cN\) be a negative support family.
	Suppose that the Stanley--Reisner ring \(\C[\Delta_{\cN}]=R_0/M_{\cN}^0\) is Cohen--Macaulay and that every facet \(F\) of \(\Delta_{\cN}\) satisfies \(\rank(A_F)=d\).
	Then \(\Tor^{\Rhat}_i \bigl(\Rhat/U,\Rhat/M_{\cN}\bigr)=0\) for \(i>0\).
	In particular, \(U\cap M_{\cN}=UM_{\cN}\), \(\fT_{\cN}=0\), and \(\fD_{\cN}(e)=0\).
	If \(\cN\) is ordered, the intrinsic coefficient space equals the ambient coefficient space.
\end{theorem}

\begin{proof}
	The case \(M_{\cN}^0=R_0\) is settled before the statement, so assume that \(M_{\cN}^0\) is proper.
	Since \(\C[\Delta_{\cN}]\) is Cohen--Macaulay, the simplicial complex \(\Delta_{\cN}\) is pure.
	Put \(q=\dim\C[\Delta_{\cN}]\).
	Every facet of \(\Delta_{\cN}\) has cardinality \(q\).
	For a facet \(F\), let \(\fp_F=\langle t_j\mid j\notin F\rangle\).
	The restrictions of the linear forms \(\theta_1,\ldots,\theta_d\) to \(R_0/\fp_F\) have coefficient matrix \(A_F\).
	Hence the rank assumption gives
	\[
		\dim\frac{R_0}{\fp_F+U_0}
		=|F|-\rank(A_F)
		=q-d.
	\]
	The minimal primes of \(I_{\Delta_{\cN}}\) are the ideals \(\fp_F\).
	Therefore,
	\[
		\dim\frac{R_0}{M_{\cN}^0+U_0}=q-d.
	\]
	Thus \(\theta_1,\ldots,\theta_d\) form a partial homogeneous system of parameters for \(\C[\Delta_{\cN}]\).
	Since \(\C[\Delta_{\cN}]\) is Cohen--Macaulay, these linear forms form a regular sequence on \(\C[\Delta_{\cN}]\).

	Since \(\rank(A)=d\), the linear forms \(\theta_1,\ldots,\theta_d\) are linearly independent, hence a regular sequence in \(R_0\).
	The Koszul complex on these forms is therefore a free resolution of \(R_0/U_0\).
	Tensoring this complex with \(R_0/M_{\cN}^0\) gives a complex whose homology vanishes in positive degrees.
	Consequently, \(\Tor^{R_0}_i \bigl(R_0/U_0,R_0/M_{\cN}^0\bigr)=0\) for \(i>0\).
	Completion at \(\langle t_1,\ldots,t_n\rangle\) is flat, so the same vanishing holds over \(\Rhat\).
	The case \(i=1\) gives
	\[
		\frac{U\cap M_{\cN}}{UM_{\cN}}=0.
	\]
	The remaining assertions follow from Theorem~\ref{thm:lattice-obstruction}.
\end{proof}

\begin{remark}
	Theorem~\ref{thm:CM-regular-sequence} proves the equality \(U\cap M_{\cN}=UM_{\cN}\) by showing that the linear forms \(\theta_1,\ldots,\theta_d\) form a regular sequence on \(R_0/M_{\cN}^0\).
	The hypotheses are sufficient but not necessary: failure of either hypothesis does not by itself imply a nonzero obstruction.
\end{remark}

\subsection{A criterion applied at each factor}

The preceding criterion yields the vanishing of the entire module \(\fT_{\cN}\).
For intrinsic perturbation, this conclusion is stronger than necessary: only the classes of \(\fT_{\cN}\) represented by multiples of \(e\) matter.
We now give a criterion that applies at each factor of \(e\) and isolates these classes.

Retain \(S_0\), \(\Phi_0\), and \(P_{\cN}^0(\bt)\) from Section~\ref{subsec:finite-boundary-presentation}.
Write \(E=I_{\bzero}\setminus K_{\cN} =\{j_1,\ldots,j_c\}\) and \(e=\prod_{a=1}^c t_{j_a}\), where the displayed ordering of \(E\) is arbitrary.
Define recursively
\begin{equation}
	P^{(0)}=P_{\cN}^0(\bt),
	\qquad
	P^{(a)}=P^{(a-1)}:t_{j_a},
	\qquad
	\fj_a=P^{(a-1)}+\langle t_{j_a}\rangle.
	\label{eq:factorwise-ideals}
\end{equation}
All the ideals in \eqref{eq:factorwise-ideals} are squarefree monomial ideals.

\begin{theorem}
	\label{thm:factorwise-regularity}
	For every \(a=1,\ldots,c\), multiplication by \(t_{j_a}\) induces an exact sequence
	\begin{equation}
		\Tor^{R_0}_1(S_0,R_0/\fj_a)
		\longrightarrow
		\frac{S_0}{\Phi_0(P^{(a)})}(-1)
		\xrightarrow{\ \cdot\ell_{j_a}\ }
		\frac{S_0}{\Phi_0(P^{(a-1)})}.
		\label{eq:factorwise-Tor-segment}
	\end{equation}
	In particular, the image of the first map is
	\begin{equation}
		\frac{
			\Phi_0(P^{(a-1)}):\ell_{j_a}
		}{
			\Phi_0(P^{(a)})
		}.
		\label{eq:factorwise-colon-error}
	\end{equation}

	Let \(\bw\) be generic, let \(\bv\) be a fake exponent, and let \(\cN\) be a negative support family.
	Suppose that, for some ordering of \(E\), for every index \(a\), either \(\fj_a=R_0\), or else the Stanley--Reisner ring \(R_0/\fj_a\) is Cohen--Macaulay and every facet \(F\) of the Stanley--Reisner complex of \(\fj_a\) satisfies \(\rank(A_F)=d\).
	Then \(\Pi_{\cN}:m_{\bv,\cN} = \Phi\bigl(P_{\cN}(\bt):e\bigr) = \Phi\bigl(Q_{\cN}(\bt):e\bigr)\).
	Consequently, \(\fD_{\cN}(e)=0\).
\end{theorem}

\begin{proof}
	Multiplication by \(t_{j_a}\) gives a short exact sequence
	\[
		0\longrightarrow
		\frac{R_0}{P^{(a)}}(-1)
		\xrightarrow{\ \cdot t_{j_a}\ }
		\frac{R_0}{P^{(a-1)}}
		\longrightarrow
		\frac{R_0}{\fj_a}
		\longrightarrow0,
	\]
	whose first map is injective precisely because \(P^{(a)}=P^{(a-1)}:t_{j_a}\).
	Tensoring with \(S_0=R_0/U_0\) gives \eqref{eq:factorwise-Tor-segment}, and the kernel of the multiplication map there is the quotient in \eqref{eq:factorwise-colon-error}.

	Assume the stated combinatorial hypotheses.
	If \(\fj_a\ne R_0\), the Cohen--Macaulay hypothesis and the rank assumption imply that the linear forms \(\theta_1,\ldots,\theta_d\) form a partial homogeneous system of parameters for \(R_0/\fj_a\), so the proof of Theorem~\ref{thm:CM-regular-sequence}, applied to \(\fj_a\), shows that these linear forms form a regular sequence on \(R_0/\fj_a\).
	Hence \(\Tor^{R_0}_i(S_0,R_0/\fj_a)=0\) for \(i>0\), and the same vanishing is immediate when \(\fj_a=R_0\).
	Formula \eqref{eq:factorwise-colon-error} therefore gives \(\Phi_0(P^{(a-1)}):\ell_{j_a}=\Phi_0(P^{(a)})\) for \(1\leq a\leq c\), and iterating these identities with \((J:f):g=J:(fg)\) gives \(\Phi_0(P^{(0)}):\prod_{a=1}^c\ell_{j_a}=\Phi_0(P^{(c)})=\Phi_0(P^{(0)}:e)\).

	All the ideals above are homogeneous, localization of \(S_0\) at its homogeneous maximal ideal followed by completion is flat, and taking the colon by a fixed element commutes with this base change, so extending the last identity to \(\Shat\) gives \(\Pi_{\cN}:m_{\bv,\cN}=\Phi(P_{\cN}(\bt):e)\).
	Since \(P_{\cN}(\bt):e\subseteq Q_{\cN}(\bt):e\), and since the inclusion proved in Theorem~\ref{thm:lattice-obstruction} gives \(\Phi(Q_{\cN}(\bt):e)\subseteq\Pi_{\cN}:m_{\bv,\cN}\), the three ideals in the display of the statement are equal.
	Theorem~\ref{thm:lattice-obstruction} now gives \(\fD_{\cN}(e)=0\).
\end{proof}

The exact sequences in Theorem~\ref{thm:factorwise-regularity} also give bounds whose proofs use neither the Cohen--Macaulay hypothesis nor the rank assumption of that theorem.

\begin{proposition}
	\label{prop:factorwise-obstruction-filtration}
	Retain the notation of Theorem~\ref{thm:factorwise-regularity}.
	If \(c=0\), then \(J_{\cN}^{\amb}=J_{\cN}^{\intr}\).
	Suppose that \(c\geq1\).
	For \(0\leq a\leq c\), define homogeneous ideals \(\fb_a\) and \(\fc_a\) of \(S_0\) by \(\fb_a=\Phi_0(P^{(a)})\) and \(\fc_a=\Phi_0(P^{(0)}):\prod_{b=1}^a\ell_{j_b}\), where the product for \(a=0\) is \(1\).
	For \(1\leq a\leq c\), define the graded module \(\fE_a\) by
	\begin{equation}
		\fE_a
		=
		\frac{\fb_{a-1}:\ell_{j_a}}{\fb_a}.
		\label{eq:factorwise-error-modules}
	\end{equation}
	Then \(\fb_a\subseteq\fc_a\), and multiplication by \(\ell_{j_a}\) gives an exact sequence
	\begin{equation}
		0
		\longrightarrow
		\fE_a(-1)
		\longrightarrow
		\left(\frac{\fc_a}{\fb_a}\right)(-1)
		\xrightarrow{\ \cdot\ell_{j_a}\ }
		\frac{\fc_{a-1}}{\fb_{a-1}}
		\longrightarrow
		\frac{\fc_{a-1}}
		{\fb_{a-1}+\ell_{j_a}\fc_a}
		\longrightarrow0.
		\label{eq:factorwise-obstruction-exact-sequence}
	\end{equation}
	Moreover, \(\fb_c=\Phi_0(P_{\cN}^0(\bt):e)\) and \(\fc_c=J_{\cN}^{\intr}\), and there is an exact sequence
	\begin{equation}
		0
		\longrightarrow
		\frac{J_{\cN}^{\amb}}{\fb_c}
		\longrightarrow
		\frac{\fc_c}{\fb_c}
		\longrightarrow
		\frac{J_{\cN}^{\intr}}
		{J_{\cN}^{\amb}}
		\longrightarrow0.
		\label{eq:factorwise-obstruction-final-quotient}
	\end{equation}
	The quotient \(J_{\cN}^{\intr}/J_{\cN}^{\amb}\) admits a filtration
	\begin{equation}
		0=\fH_0
		\subseteq\fH_1
		\subseteq\cdots
		\subseteq\fH_c
		=
		\frac{J_{\cN}^{\intr}}
		{J_{\cN}^{\amb}}
		\label{eq:factorwise-obstruction-filtration}
	\end{equation}
	such that \(\fH_k/\fH_{k-1}\) is a subquotient of \(\fE_{c-k+1}(k-1)\) for \(1\leq k\leq c\).
	Consequently,
	\begin{equation}
		\Supp_{S_0}
		\left(
		\frac{J_{\cN}^{\intr}}
		{J_{\cN}^{\amb}}
		\right)
		\subseteq
		\bigcup_{a=1}^c\Supp_{S_0}(\fE_a)
		\subseteq
		\bigcup_{a=1}^c
		\Supp_{S_0}
		\bigl(
		\Tor^{R_0}_1(S_0,R_0/\fj_a)
		\bigr),
		\label{eq:factorwise-obstruction-support}
	\end{equation}
	and for every \(i\geq0\) we have the degree bound
	\begin{equation}
		\dim_{\C}
		\left(
		\frac{J_{\cN}^{\intr}}
		{J_{\cN}^{\amb}}
		\right)_i
		\leq
		\sum_{a=1}^c
		\dim_{\C}(\fE_a)_{i+c-a}.
		\label{eq:factorwise-obstruction-degree-bound}
	\end{equation}
	If all the modules \(\fE_a\) have finite length, then
	\begin{equation}
		\dim_{\C}\fD_{\cN}(e)
		\leq
		\sum_{a=1}^c\dim_{\C}\fE_a.
		\label{eq:factorwise-obstruction-length-bound}
	\end{equation}
	If \(\cN\) is ordered, the left-hand side of \eqref{eq:factorwise-obstruction-degree-bound} is the dimension of the degree-\(i\) ambient coefficient space modulo the degree-\(i\) intrinsic coefficient space.
\end{proposition}

\begin{proof}
	If \(c=0\), then \(e=1\) and \(M_{\cN}^0=R_0\), so \(Q_{\cN}^0(\bt)=U_0+P_{\cN}^0(\bt)\) and \(J_{\cN}^{\amb}=J_{\cN}^{\intr}\); suppose that \(c\geq1\).
	Since \((\prod_{b\leq a}t_{j_b})P^{(a)}\subseteq P^{(0)}\), applying \(\Phi_0\) gives \(\fb_a\subseteq\fc_a\), and associativity of the colon operation gives \(\fc_a=\fc_{a-1}:\ell_{j_a}\).
	The inclusion \(t_{j_a}P^{(a)}\subseteq P^{(a-1)}\) makes multiplication by \(\ell_{j_a}\) the middle map in \eqref{eq:factorwise-obstruction-exact-sequence}, and \(\fb_{a-1}\subseteq\fc_{a-1}\) gives \(\fb_{a-1}:\ell_{j_a}\subseteq\fc_a\), so the kernel of that map is \(\fE_a(-1)\) and its cokernel is the last quotient in \eqref{eq:factorwise-obstruction-exact-sequence}.
	Successive colons give \(P^{(c)}=P_{\cN}^0(\bt):e\), hence \(\fc_c=J_{\cN}^{\intr}\), and \(P_{\cN}^0(\bt)\subseteq Q_{\cN}^0(\bt)\) with Theorem~\ref{thm:finite-boundary-classification} give \(\fb_c\subseteq J_{\cN}^{\amb}\subseteq J_{\cN}^{\intr}\) and hence \eqref{eq:factorwise-obstruction-final-quotient}.

	Put \(\fF_a=\fc_a/\fb_a\) and let \(\fK_a\subseteq\fF_{a-1}\) be the image of the middle map, so that shifting the kernel sequence gives \(0\to\fE_a\to\fF_a\to\fK_a(1)\to0\).
	We start from \(\fF_0=0\) and, for each \(a\), intersect a filtration of \(\fF_{a-1}\) with \(\fK_a\), shift, and adjoin \(\fE_a\); this construction gives inductively a filtration of \(\fF_a\) whose \(k\)-th factor is a subquotient of \(\fE_{a-k+1}(k-1)\).
	Passing to the quotient in \eqref{eq:factorwise-obstruction-final-quotient} gives \eqref{eq:factorwise-obstruction-filtration}.
	Shifts do not change the support, so the filtration gives the first inclusion in \eqref{eq:factorwise-obstruction-support}, and \eqref{eq:factorwise-colon-error} identifies \(\fE_a\) with the image of the first map in \eqref{eq:factorwise-Tor-segment}, which gives the second.
	Taking degree-\(i\) components gives \eqref{eq:factorwise-obstruction-degree-bound}, taking lengths in the finite-length case gives \eqref{eq:factorwise-obstruction-length-bound} because completion preserves the length of \(J_{\cN}^{\intr}/J_{\cN}^{\amb}\), and the last assertion follows from Proposition~\ref{prop:defect-coefficient-duality}.
\end{proof}

For matroid connectivity, we follow the conventions of \cite[Chapter~8]{Oxl11}.
The criterion of Theorem~\ref{thm:factorwise-regularity} takes the following form when every ideal \(\fj_a\) is either \(R_0\) or generated by two monomials.

\begin{corollary}
	\label{cor:principal-boundary-flag}
	Retain the notation of Theorem~\ref{thm:factorwise-regularity}.
	Suppose that \(E\) has an ordering for which, for every \(a\), either \(\fj_a=R_0\) or
	\begin{equation}
		\fj_a
		=
		\left\langle
		t_{j_a},\ \prod_{k\in W_a}t_k
		\right\rangle
		\label{eq:principal-boundary-slice}
	\end{equation}
	for a nonempty set \(W_a\subseteq\{1,\ldots,n\}\setminus\{j_a\}\).
	If \(\rank(A_{\{1,\ldots,n\}\setminus\{j_a,k\}})=d\) for \(k\in W_a\), then \(\fD_{\cN}(e)=0\).

	When \(\rank L=2\) and the column matroid of \(A\) is three-connected, the rank condition is automatic.
	Thus, in that case, condition \eqref{eq:principal-boundary-slice} alone implies \(\fD_{\cN}(e)=0\).
\end{corollary}

\begin{proof}
	If the ideal \eqref{eq:principal-boundary-slice} is proper, then the quotient \(R_0/\fj_a\) is a hypersurface ring in the variables other than \(t_{j_a}\).
	The quotient \(R_0/\fj_a\) is therefore Cohen--Macaulay.
	The facets of the Stanley--Reisner complex of \(\fj_a\) are the sets \(\{1,\ldots,n\}\setminus\{j_a,k\}\) for \(k\in W_a\).
	The first assertion follows from Theorem~\ref{thm:factorwise-regularity}.

	For the last assertion, recall that duality preserves three-connectivity.
	A three-connected matroid of rank 2 is simple, and hence is the uniform matroid of rank 2 on \(n\) elements.
	The dual of the column matroid of \(A\) is therefore the uniform matroid of rank 2 on \(n\) elements, so the column matroid of \(A\) is the uniform matroid of rank \(n-2\) on \(n\) elements.
	Since \(d=n-2\), every submatrix \(A_{\{1,\ldots,n\}\setminus\{j_a,k\}}\) has rank \(d\).
\end{proof}

\begin{remark}
	The hypotheses of Theorem~\ref{thm:CM-regular-sequence} make \(\theta_1,\ldots,\theta_d\) a regular sequence on \(R_0/M_{\cN}^0\).
	The hypotheses of Theorem~\ref{thm:factorwise-regularity} make these linear forms a regular sequence only on the successive quotients \(R_0/\fj_a\).
	Theorem~\ref{thm:factorwise-regularity} does not assume that the module \(\fT_{\cN}\) vanishes; the criterion tests only the classes represented by multiples of \(e\).
	Formula \eqref{eq:factorwise-colon-error} identifies, at each factor, the image of the map out of the \(\Tor\) term, and this image is what can give rise to such a class.
\end{remark}

\section{Low lattice rank}
\label{sec:low-lattice-rank}

In this section, we work in low lattice rank and distinguish ordered negative support families from the distinguished collection.

\subsection{Lattice rank 1}

A proper ordered family can already have a nonzero obstruction in lattice rank 1, but the distinguished collection cannot.

\begin{theorem}
	\label{thm:rank-one-ordered-counterexample}
	There exist a homogeneous \(A\)-hypergeometric system with \(\rank L=1\), a generic direction, a fake exponent \(\bv\), and a proper ordered negative support family \(\cN\subsetneq\cN_{\bv}\) such that \(J_{\cN}^{\amb}=\langle s_1\rangle\) and \(J_{\cN}^{\intr}=\C[s_1]\), and hence \(\fD_{\cN}(e)\simeq\C[s_1]/\langle s_1\rangle\).
	In particular, the obstruction has dimension one, the intrinsic coefficient space is zero, and the ambient coefficient space has dimension one.
\end{theorem}

\begin{proof}
	Set
	\[
		A=
		\begin{bmatrix}
			1  & 1 & 1 & 1 & 1 \\
			2  & 1 & 0 & 0 & 0 \\
			2  & 0 & 1 & 0 & 0 \\
			-1 & 0 & 0 & 1 & 0
		\end{bmatrix},
		\qquad
		\bb=\tp{(1,-2,-2,1,2)},
	\]
	and take \(\bw=(1,0,0,0,0)\), \(\bv=\tp{(-2,-2,0,-1,0)}\), and \(\bbeta=A\bv=\tp{(-5,-6,-4,1)}\).
	The equality \(A\bb=0\) and the unimodular minor on columns \(2,3,4,5\) show that \(L=\Z\bb\), and the first row of \(A\) shows that \(A\) is homogeneous.
	Since \(\bb\) is primitive, \(I_A=\langle\underline{\partial_1\partial_4\partial_5^2}-\partial_2^2\partial_3^2\rangle\), where the underlined monomial is the initial monomial with respect to \(\bw\).
	The distraction of that monomial vanishes at \(\bv\) because \(v_5=0\), so \(\bv\) is a fake exponent.
	For \(k\in\Z\), direct substitution gives \(\nsupp(\bv+k\bb)=\{1,4,5\}\) for \(k\leq-1\), \(\{1,2,4\}\) for \(k=0\), \(\{1,2,3\}\) for \(k=1\), and \(\{2,3\}\) for \(k\geq2\).
	Since the weight of \(k\bb\) is \(k\), the distinguished collection is \(\cN_{\bv}=\{\{1,2,4\},\{1,2,3\},\{2,3\}\}\), and the same classification shows that the proper subfamily \(\cN=\{\{1,2,4\},\{2,3\}\}\) is ordered.
	The definitions give \(K_{\cN}=\{2\}\), \(e=t_1t_4\), \(M_{\cN}=\langle t_1t_4,t_3\rangle\), and \(P_{\cN}(\bt)=\langle t_1t_3,t_1t_4t_5\rangle\).
	Since \(\Phi\) sends \(t_1,t_2,t_3,t_4,t_5\) to \(s_1,-2s_1,-2s_1,s_1,2s_1\), respectively, we obtain \(\Pi_{\cN}=\langle s_1^2\rangle\), \(m_{\bv,\cN}=s_1^2\), and \(J_{\cN}^{\intr}=\C[s_1]\).
	Recall that \(U=\langle A\bt\rangle=\ker\Phi\) and \(Q_{\cN}(\bt)=UM_{\cN}+P_{\cN}(\bt)\).
	Since \(e\in M_{\cN}\) and \(et_5\in P_{\cN}(\bt)\), the ideal \(\Phi(Q_{\cN}(\bt):e)\) contains \(s_1\).
	Since the degree-two part of \(Q_{\cN}(\bt)\) is contained in \(\langle t_3\rangle\), we have \(e\notin Q_{\cN}(\bt)\), so \(\Phi(Q_{\cN}(\bt):e)\) is proper.
	The homogeneous ideal \(\Phi(Q_{\cN}(\bt):e)\) equals \(\langle s_1\rangle\Shat\), and hence \(J_{\cN}^{\amb}=\langle s_1\rangle\subseteq S_0\).
	The assertions now follow from Theorem~\ref{thm:finite-boundary-classification} and Proposition~\ref{prop:defect-coefficient-duality}.
\end{proof}

We now turn to the distinguished collection.
Let \(L=\Z\bb\), where \(\bb\) is primitive and has its sign chosen so that \(\bw\cdot\bb>0\).
Put \(I_k=\nsupp(\bv+k\bb)\) for \(k\in\Z\); in particular, \(I_0=I_{\bzero}\).
Since \(L=\Z\bb\), every binomial \(\bpartial^{(k\bb)_+}-\bpartial^{(k\bb)_-}\) is divisible by \(\bpartial^{\bb_+}-\bpartial^{\bb_-}\), so \(I_A\) is generated by the single binomial \(\bpartial^{\bb_+}-\bpartial^{\bb_-}\), whose initial monomial is \(\bpartial^{\bb_+}\) because \(\bw\cdot\bb>0\).
Thus \(\cC(\bw)=\N\bb\).
The distinguished collection \(\cN_{\bv}\) consists of the negative supports whose support fibers have nonnegative weight.

The following lemma determines these support fibers.

\begin{lemma}
	\label{lem:rank-one-sign-cells}
	Suppose that \(L=\Z\bb\), that \(\bw\cdot\bb>0\), and that \(\bv\) is a fake exponent.
	For every realized negative support \(I\), the set \(\{k\in\Z\mid I_k=I\}\) is an interval in \(\Z\).
	The support fiber of \(I_{\bzero}\) has minimum \(0\), the distinguished collection is \(\cN_{\bv}=\{I_k\mid k\geq0\}\), and \(\cN_{\bv}\) is ordered.
\end{lemma}

\begin{proof}
	If \(v_j\notin\Z\), then the index \(j\) belongs to none of the sets \(I_k\).
	Suppose that \(v_j\in\Z\).
	If \(b_j>0\), the condition \(j\in I_k\) defines a left half-line in \(\Z\); if \(b_j<0\), the condition defines a right half-line; and if \(b_j=0\), the condition holds either for every \(k\) or for no \(k\).
	Hence the fiber of every negative support is an intersection of integer half-lines, and such an intersection is an interval.

	In lattice rank 1, the initial monomial of the toric ideal is \(\bpartial^{\bb_+}\).
	Since \(\bv\) is a fake exponent, there is an index \(j\) such that \(b_j>0\) and \(v_j\in\{0,\ldots,b_j-1\}\).
	Thus \(j\notin I_{\bzero}\) and \(j\in I_{-1}\).
	The support fiber of \(I_{\bzero}\) contains \(0\) but does not contain \(-1\), so its minimum is \(0\).

	If a support fiber met both \(\Z_{<0}\) and \(\Z_{\geq0}\), it would contain \(0\).
	It would then be the fiber of \(I_{\bzero}\), which contradicts the fact that this fiber has minimum \(0\).
	Hence every fiber that meets \(\Z_{\geq0}\) is contained in \(\Z_{\geq0}\).
	Since \(\bw\cdot\bb>0\), the definition of \(\cN_{\bv}\) gives the asserted description of the distinguished collection.
	The index found above belongs to every support \(I_k\) with \(k<0\) and to no support \(I_k\) with \(k\geq0\).
	Thus no support outside \(\cN_{\bv}\) is contained in a member of \(\cN_{\bv}\), which proves that \(\cN_{\bv}\) is ordered.
\end{proof}

\begin{theorem}
	\label{thm:rank-one-full-support}
	Suppose that \(\rank L=1\), let \(\bv\) be a fake exponent, and take \(\cN=\cN_{\bv}\).
	Then \(\Phi(Q_{\cN}(\bt):e)=\Pi_{\cN}:m_{\bv,\cN}\).
	Equivalently, \(\fD_{\cN}(e)=0\).
	The distinguished collection is ordered, and its intrinsic coefficient space equals its ambient coefficient space.
\end{theorem}

\begin{proof}
	By Lemma~\ref{lem:rank-one-sign-cells}, we have \(\cN_{\bv}=\{I_k\mid k\geq0\}\).

	Set \(K=\bigcap_{k\geq0}I_k\), \(E=I_0\setminus K\), and \(H=I_{-1}\setminus I_0\).
	An index \(j\in E\) has \(b_j>0\): if an index \(j\) with \(b_j<0\) belongs to \(I_0\), then \(j\) belongs to \(I_k\) for every \(k\geq0\), and an index \(j\) with \(b_j=0\) belongs either to every \(I_k\) or to none.
	The same argument gives \(b_j>0\) for every \(j\in H\).
	In particular, none of the linear forms \(\Phi(t_j)=b_js_1\), with \(j\in E\cup H\), is zero.

	For \(k<0\), every index \(j\in I_{-1}\) with \(b_j>0\) belongs to \(I_k\).
	If an index \(j\) with \(b_j<0\) belongs to \(I_{-1}\), then \(j\) belongs to \(I_k\) for every \(k\geq-1\), hence belongs to \(K\), and is therefore absent from \(I_{-1}\setminus K\).
	Thus \((I_{-1}\setminus K)=E\sqcup H \subseteq I_k\setminus K\) for \(k<0\).
	For \(k\geq0\), the set of indices \(j\in I_k\setminus K\) with \(b_j>0\) is contained in \(E\).
	Therefore every generator of \(P_{\cN}(\bt)\) is divisible by \(\bt^{E\sqcup H}\), and the pair \((I_0,I_{-1})\) contributes exactly this monomial.
	Hence \(P_{\cN}(\bt)=\langle\bt^{E\sqcup H}\rangle\) and \(e=\bt^E\).

	Since \(\Phi(t_j)=b_js_1\), there are nonzero constants \(c_E,c_H\) such that \(m_{\bv,\cN}=c_Es_1^{|E|}\) and \(\Pi_{\cN}=\langle c_Ec_Hs_1^{|E|+|H|}\rangle\).
	It follows that \(\Pi_{\cN}:m_{\bv,\cN} =\langle s_1^{|H|}\rangle\).
	On the other hand, \(\bt^H \in P_{\cN}(\bt):e \subseteq Q_{\cN}(\bt):e\), and \(\Phi(\bt^H)=c_Hs_1^{|H|}\).
	Since \(\Phi(UM_{\cN})=0\), the inclusion \(\Phi(Q_{\cN}(\bt):e)\subseteq \Pi_{\cN}:m_{\bv,\cN}\) holds in general; the computation above gives the reverse inclusion and hence the equality.
\end{proof}

The following lemma extends this argument from a fake exponent to an arbitrary vector \(\ba\in\C^n\).

\begin{lemma}
	\label{lem:rank-one-tail}
	Let \(L=\Z\bb\), choose the sign of \(\bb\) so that \(\bw\cdot\bb>0\), and, for a vector \(\ba\in\C^n\), put \(J_k=\nsupp(\ba+k\bb)\) for \(k\in\Z\).
	Suppose that the fiber of \(J_0\) has minimum \(0\).
	For the collection \(\cW=\{J_k\mid k\geq0\}\), the obstruction vanishes.
	Here the obstruction is defined with the set \(\{J_k\mid k\in\Z\}\) of all realized negative supports in place of \(\sS(\bv)\).
\end{lemma}

\begin{proof}
	By the first paragraph of the proof of Lemma~\ref{lem:rank-one-sign-cells}, every support fiber is an interval.
	If a fiber met both \(\Z_{<0}\) and \(\Z_{\geq0}\), it would contain \(0\), and hence would be the fiber of \(J_0\).
	This contradicts the assumption that the fiber of \(J_0\) has minimum \(0\).
	Thus \(\cW\) is precisely the collection of supports whose fibers are contained in the nonnegative half-line.

	Put \(K=\bigcap_{k\geq0}J_k\), \(E=J_0\setminus K\), and \(H=J_{-1}\setminus J_0\).
	The proof of Theorem~\ref{thm:rank-one-full-support} uses only this interval description.
	That proof gives \(P_{\cW}(\bt)=\langle\bt^{E\sqcup H}\rangle\) and \(e=\bt^E\), and therefore \(\Pi_{\cW}:m_{\ba,\cW}=\langle s_1^{|H|}\rangle=\Phi(Q_{\cW}(\bt):e)\).
	The obstruction vanishes by Theorem~\ref{thm:lattice-obstruction}.
\end{proof}

These results yield the following corollary.

\begin{corollary}
	\label{cor:minimal-lattice-codimension}
	Lattice rank 1 is the least rank in which a nonzero obstruction for a negative support family can occur, and the ordered family in Theorem~\ref{thm:rank-one-ordered-counterexample} has a nonzero obstruction in that rank.
	For the distinguished collection, the obstruction vanishes whenever the lattice rank is at most 1.
	The distinguished collection in the five-column example of Theorem~\ref{thm:lattice-counterexample} has a nonzero obstruction in lattice rank 2, but that collection is not ordered.
	The distinguished collection in Theorem~\ref{thm:ordered-distinguished-counterexample} is ordered and also has a nonzero obstruction in lattice rank 2.
\end{corollary}

\begin{proof}
	If \(\rank L=0\), then \(I_{\bzero}\) is the only realized negative support, \(K_{\cN}=I_{\bzero}\), \(e=1\), and \(P_{\cN}(\bt)=0\); both ideals in \eqref{eq:defect-image} are zero.
	Theorem~\ref{thm:rank-one-ordered-counterexample} gives the example with an ordered negative support family in lattice rank 1, and Theorem~\ref{thm:rank-one-full-support} proves the assertion for the distinguished collection in lattice rank 1.
	The matrix in \eqref{eq:lattice-counterexample-AB} has lattice rank 2.
\end{proof}

In the next corollary, we use the notation \(\alpha\), \(\beta\), \(g_0\), \(h_0\), \(\omega\), \(\delta_{\omega}\), \(d_h\), \(d_{\delta}\), and \(d_{h,\delta}\) from Theorem~\ref{thm:two-minimal-support-formula} and Corollary~\ref{cor:two-support-principal-formula}.
Combining Corollary~\ref{cor:minimal-lattice-codimension} with Corollary~\ref{cor:two-support-principal-formula}, we obtain the following.

\begin{corollary}
	\label{cor:distinguished-two-support-complete-intersection}
	Let \(\bv\) be a fake exponent, and take \(\cN=\cN_{\bv}\).
	Suppose that \(\cG_{\cN}=\{G_1,G_2\}\) and \(E=G_1\), and suppose that there are nonzero homogeneous polynomials \(\delta\) and \(\gamma\) such that, for some choice of \(\iota\) in \eqref{eq:boundary-generator-relations}, the ideals in \eqref{eq:two-support-multiplier-ideals} satisfy \(\fa_1=\langle\delta\rangle\) and \(\fa_2=\langle\gamma\rangle\).
	If \(\fD_{\cN}(e)\ne0\), then \(\rank L=2\), \(\gcd(\delta,\omega)=1\), and \(\gcd(h_0,\delta)=1\).
	In that case, \(\delta_{\omega}=\delta\) and \(d_{h,\delta}=0\).
	Moreover, \(d_h,d_{\delta}\geq1\) and \(\dim_{\C}\fD_{\cN}(e)=d_hd_{\delta}\), and
	\begin{align}
		\frac{J_{\cN}^{\intr}}{J_{\cN}^{\amb}} &\simeq \frac{S_0}{\langle h_0,\delta\rangle}(-\deg\omega),
		\label{eq:distinguished-two-support-model}\\
		\Hilb\bigl(\fD_{\cN}(e);u\bigr) &= u^{\deg\omega}(1+u+\cdots+u^{d_h-1})(1+u+\cdots+u^{d_{\delta}-1}).
		\label{eq:distinguished-two-support-Hilbert}
	\end{align}
\end{corollary}

\begin{proof}
	Theorem~\ref{thm:finite-boundary-classification} shows that \(S_0/J_{\cN}^{\amb}\) is Artinian.
	Since the obstruction is nonzero, the ideal \(J_{\cN}^{\amb}\) is proper.
	Corollary~\ref{cor:two-support-principal-formula} gives \(J_{\cN}^{\amb}=\langle\delta,h_0\omega\rangle\).
	Krull's height theorem gives \(r\leq2\).
	Corollary~\ref{cor:minimal-lattice-codimension} gives \(r\geq2\), and hence \(r=2\).
	The proper Artinian ideal \(\langle\delta,h_0\omega\rangle\) has height 2, so its two generators are relatively prime.
	Thus \(\gcd(\delta,\omega)=1\) and \(\gcd(h_0,\delta)=1\).
	The vanishing statement in Corollary~\ref{cor:two-support-principal-formula} gives \(d_h,d_{\delta}\geq1\).
	Equation \eqref{eq:two-support-gcd-model} now gives \eqref{eq:distinguished-two-support-model}.
	The polynomials \(h_0\) and \(\delta\) form a homogeneous regular sequence in \(S_0\).
	The quotient of \(S_0\) by \(\langle h_0,\delta\rangle\) has Hilbert series
	\[
		\frac{(1-u^{d_h})(1-u^{d_{\delta}})}{(1-u)^2}
		=
		(1+u+\cdots+u^{d_h-1})(1+u+\cdots+u^{d_{\delta}-1}).
	\]
	Applying the degree shift by \(-\deg\omega\) proves \eqref{eq:distinguished-two-support-Hilbert}, and evaluating the Hilbert series at \(u=1\) gives \(\dim_{\C}\fD_{\cN}(e)=d_hd_{\delta}\).
\end{proof}

\subsection{Proportional rows of a Gale dual in lattice rank 2}

The rank 2 counterexample in Theorem~\ref{thm:lattice-counterexample} has two proportional rows in its Gale dual.
We now give a nonproportionality condition on pairs of rows of \(B\) under which the obstruction vanishes.

Let \(\bb_j\) denote the \(j\)-th row of an \(n\times2\) Gale dual \(B\).
For \(\bv\in\C^n\), put \(J_{\Z}(\bv)=\{j\mid v_j\in\Z\}\), and define \(J_{\mov}(\bv)\) by
\begin{equation}
	J_{\mov}(\bv)
	=
	\{j\in J_{\Z}(\bv)\mid \bb_j\ne\bzero\}.
	\label{eq:moving-integral-set}
\end{equation}
Also put \(J_{\mathrm{const}}^-(\bv)=\{j\in J_{\Z}(\bv)\mid \bb_j=\bzero,\ v_j<0\}\).
The set \(J_{\mov}(\bv)\) does not depend on the choice of the Gale dual \(B\).
No index outside \(J_{\Z}(\bv)\) belongs to a negative support.
An index \(j\) with \(\bb_j=\bzero\) either belongs to every negative support or belongs to none.

\begin{lemma}
	\label{lem:rank-two-affine-support-bound}
	Suppose that \(\rank L=2\), and put \(p=\#J_{\mov}(\bv)\).
	Then
	\begin{equation}
		\#\sS(\bv)
		\leq
		\cR_p=1+\frac{p(p+1)}2.
		\label{eq:rank-two-support-bound}
	\end{equation}
	More precisely, consider the affine line arrangement
	\begin{equation}
		\bb_j\bz+v_j+\frac12=0
		\qquad
		(j\in J_{\mov}(\bv)),
		\label{eq:rank-two-affine-arrangement}
	\end{equation}
	and label each region by adjoining \(J_{\mathrm{const}}^-(\bv)\) to the set of indices \(j\in J_{\mov}(\bv)\) whose lines have the region on their negative side.
	The realized negative supports are exactly the labels of the regions that contain an integral point.
\end{lemma}

\begin{proof}
	If \(j\notin J_{\Z}(\bv)\), then \(v_j+\bb_j\bz\) is nonintegral for every \(\bz\in\Z^2\), so \(j\) belongs to no negative support, and if \(\bb_j=\bzero\), then \(j\) belongs to every negative support precisely when \(j\in J_{\mathrm{const}}^-(\bv)\).
	For every remaining index, both \(v_j\) and \(\bb_j\bz\) are integral, so
	\begin{equation}
		j\in\nsupp(\bv+B\bz)
		\quad\Longleftrightarrow\quad
		\bb_j\bz+v_j+\frac12<0.
		\label{eq:rank-two-affine-sign-test}
	\end{equation}
	No integral point lies on a line in \eqref{eq:rank-two-affine-arrangement}, so every integral point lies in a region whose label is its negative support, and a region label is realized precisely when the region contains an integral point.
	Adding the \(k\)-th line to an arrangement of \(k-1\) lines creates at most \(k\) new regions, so starting from one region gives at most \(1+\sum_{k\leq p}k=1+p(p+1)/2\) regions, and coincident or parallel lines only decrease this number.
\end{proof}

Suppose that \(v_j\in\{-1,0\}\) for \(j\in J_{\mov}(\bv)\).
For \(j\in J_{\mov}(\bv)\), put \(\widehat{\bb}_j=-\bb_j\) if \(j\in I_{\bzero}\) and \(\widehat{\bb}_j=\bb_j\) otherwise.
Put
\begin{equation}
	\cK_0(\bv,B)
	=
	\left\{
	\bz\in\R^2
	\ \middle|\
	\widehat{\bb}_j\bz\geq0
	\text{ for }j\in J_{\mov}(\bv)
	\right\}.
	\label{eq:base-sign-cone}
\end{equation}

\begin{remark}
	Put \(\bc=\bw B\).
	The condition used below is
	\begin{equation}
		\bc\not\parallel\bb_j
		\qquad
		(j\in J_{\mov}(\bv)).
		\label{eq:moving-direction-genericity}
	\end{equation}
	Here \(\not\parallel\) means that the two row vectors are linearly independent.
	After \(J_{\mov}(\bv)\) is fixed, this condition holds outside a finite union of proper hyperplanes in every full-dimensional Gr\"obner cone.
	This condition is weaker than requiring \(\bc\bz\ne0\) for every nonzero \(\bz\in\Z^2\).
	In fact, a nonzero rational row vector can avoid the finitely many directions of the rows of \(B\) and still vanish on a nonzero vector in \(\Z^2\).
\end{remark}

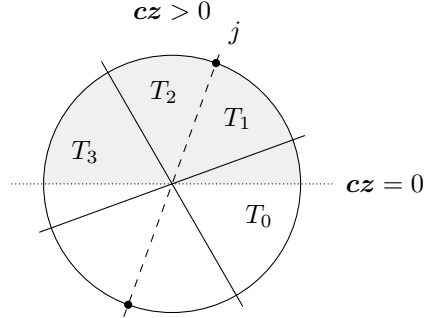
\begin{figure}[ht]
	\centering
	\begin{tikzpicture}[scale=1.7]
		\fill[gray!12] (0,0) -- (0:1) arc (0:180:1) -- cycle;
		\draw (0,0) circle (1);
		\draw[densely dotted] (-1.25,0) -- (1.25,0);
		\draw (200:1.1) -- (20:1.1);
		\draw[dashed] (250:1.1) -- (70:1.1);
		\draw (300:1.1) -- (120:1.1);
		\fill (70:1) circle (0.9pt);
		\fill (250:1) circle (0.9pt);
		\node at (45:0.72) {\(T_1\)};
		\node at (95:0.72) {\(T_2\)};
		\node at (340:0.72) {\(T_0\)};
		\node at (160:0.72) {\(T_3\)};
		\node[right] at (1.28,0) {\(\bc\bz=0\)};
		\node at (90:1.34) {\(\bc\bz>0\)};
		\node[above right] at (70:1.08) {\(j\)};
	\end{tikzpicture}
	\caption{The central line arrangement with three lines in the proofs of Lemmas~\ref{lem:central-endpoint-union} and \ref{lem:planar-principal-flag}: the regions \(T_1\) and \(T_2\) lie in the open semicircle of directions with \(\bc\bz>0\), the regions \(T_0\) and \(T_3\) are adjacent across the two ends of that semicircle, and the two antipodal directions on the dashed \(j\)-th line are marked.}
	\label{fig:central-arrangement}
\end{figure}

We prepare a lemma on central line arrangements; see Figure~\ref{fig:central-arrangement}.

\begin{lemma}
	\label{lem:central-endpoint-union}
	Let distinct central lines in \(\R^2\) be indexed by a finite set \(J\), for each \(j\in J\) choose an open half-plane bounded by the \(j\)-th line, and label each region by the set of indices whose chosen half-planes contain that region.
	Let \(T_1,\ldots,T_m\) be the consecutive regions contained in an open half-plane whose boundary contains no ray of the arrangement.
	Denote by \(T_0\) and \(T_{m+1}\) the regions adjacent across the initial and terminal boundary rays of that half-plane, respectively.
	The label of each region \(T_i\) is also denoted by \(T_i\).
	Distinct regions have distinct labels.

	Fix \(j\in J\).
	Suppose that the regions among \(T_1,\ldots,T_m\) whose labels omit \(j\) form a nonempty initial segment and that \(T_0\) also omits \(j\).
	Let \(I\) be the label of one of \(T_1,\ldots,T_m\), and let \(J'\) be the label of a region not among \(T_1,\ldots,T_m\).
	If both labels omit \(j\), then
	\begin{equation}
		T_0\cup T_1\subseteq I\cup J'.
		\label{eq:initial-endpoint-union}
	\end{equation}
	The analogous statement holds at the terminal end: if the regions among \(T_1,\ldots,T_m\) whose labels omit \(j\) form a nonempty terminal segment and \(T_{m+1}\) omits \(j\), then
	\begin{equation}
		T_m\cup T_{m+1}\subseteq I\cup J'.
		\label{eq:terminal-endpoint-union}
	\end{equation}
	Both conclusions remain valid after the same fixed set is adjoined to every region label.
\end{lemma}

\begin{proof}
	Two distinct regions are separated by at least one of the lines, so their labels differ in the index of that line.
	The regions whose labels omit \(j\) are exactly the regions contained in the open semicircle \(C_j\) of directions outside the half-plane chosen for \(j\).
	Every central line other than the \(j\)-th line has exactly one of its two rays inside \(C_j\).
	Hence, along any arc inside \(C_j\), the membership of an index \(k\ne j\) in the region labels changes at most once.

	The regions \(T_0\) and \(T_1\) omit \(j\), and they are adjacent across the initial boundary ray of the open half-plane containing \(T_1,\ldots,T_m\).
	That ray is not contained in the \(j\)-th line, because the boundary of that half-plane contains no ray of the arrangement; hence that ray lies in \(C_j\), and the antipodal terminal boundary ray does not.
	The regions omitting \(j\) therefore form a block of consecutive regions that crosses the initial boundary ray and not the terminal one: inside that half-plane they are the regions of the given initial segment, and outside that half-plane they begin with \(T_0\).
	Order this block linearly, starting outside that half-plane; then the region labeled \(J'\) precedes or equals \(T_0\), and \(T_1\) precedes or equals the region labeled \(I\).
	Choose an arc inside \(C_j\) from the region labeled \(J'\) to the region labeled \(I\); this arc passes through \(T_0\) and \(T_1\).

	Let \(k\in T_0\cup T_1\), and note that \(k\ne j\) because \(T_0\) and \(T_1\) omit \(j\).
	If \(k\) belonged to neither \(I\) nor \(J'\), the membership of \(k\) would change at least twice along the chosen arc: once between the region labeled \(J'\) and the first of \(T_0,T_1\) containing \(k\), and once between the last of \(T_0,T_1\) containing \(k\) and the region labeled \(I\).
	This contradicts the bound established in the first paragraph of this proof.
	Hence \(k\in I\cup J'\), which proves \eqref{eq:initial-endpoint-union}.
	Reversing the circular order proves \eqref{eq:terminal-endpoint-union}.
	Adjoining the same fixed set to every label adds that set to both sides of \eqref{eq:initial-endpoint-union} and \eqref{eq:terminal-endpoint-union}.
\end{proof}

By this lemma, we obtain the following.

\begin{lemma}
	\label{lem:planar-principal-flag}
	Suppose that \(\rank L=2\), that \(\bv\) is a fake exponent, that \(v_j\in\{-1,0\}\) for \(j\in J_{\mov}(\bv)\), that the rows of \(B\) indexed by \(J_{\mov}(\bv)\) are pairwise nonproportional, and that \(\bc=\bw B\) satisfies \eqref{eq:moving-direction-genericity}.
	Put \(p=\#J_{\mov}(\bv)\).
	Then \(p\geq2\), and, for the distinguished collection \(\cN_{\bv}\), we have
	\[
		\#\cN_{\bv}
		=
		\begin{cases}
			p-1, & \cK_0(\bv,B)\text{ has nonempty interior}, \\
			p,   & \cK_0(\bv,B)=\{\bzero\}.
		\end{cases}
	\]
	Moreover, the set \(E=I_{\bzero}\setminus K_{\cN_{\bv}}\) admits an ordering satisfying condition \eqref{eq:principal-boundary-slice}.
\end{lemma}

\begin{proof}
	Write \(\bc=\bw B\in\R^2\), and put \(J=J_{\mov}(\bv)\) and \(F=I_{\bzero}\setminus J\).
	Every index in \(F\) has a zero row of \(B\) and belongs to every realized negative support.
	Every index outside \(J_{\Z}(\bv)\) belongs to no realized negative support.
	Hence, for \(\bz\in\Z^2\), integrality gives
	\begin{equation}
		\nsupp(\bv+B\bz)
		=
		F\cup
		\left(
		(I_{\bzero}\cap J)\mathbin{\triangle}
		\left\{
		j\in J\ \middle|\ \widehat{\bb}_j\bz<0
		\right\}
		\right),
		\label{eq:binary-central-signs}
	\end{equation}
	where \(\triangle\) denotes symmetric difference.
	Label each region of the central line arrangement \(\widehat{\bb}_j\bz=0\) for \(j\in J\) by the value that the right-hand side of \eqref{eq:binary-central-signs} takes on that region.
	The hypotheses on the rows make these lines distinct, so consecutive regions have labels differing in exactly one index, distinct regions have distinct labels by convexity, and each index in \(J\) changes its membership in the region labels exactly once along any open semicircle.
	A nonzero point on the line \(\widehat{\bb}_j\bz=0\) carries the label of the adjacent region on the side \(\widehat{\bb}_j\bz>0\), and every region is a nonempty rational open cone and hence contains a lattice point; the supports realized at nonzero lattice points are therefore exactly the region labels.

	Because \(\bv\) is a fake exponent, every lattice point in the support fiber of \(I_{\bzero}\) has nonnegative \(\bc\)-weight.
	For \(p=0\), that fiber is all of \(\Z^2\), and for \(p=1\), it consists of the lattice points in a rational closed half-plane; genericity and \eqref{eq:moving-direction-genericity} then supply a lattice point of negative \(\bc\)-weight in that fiber, and hence \(p\geq2\).
	A region label belongs to \(\cN_{\bv}\) precisely when the region is contained in the open half-plane \(\bc\bz>0\): the boundary rays of such a region also have positive \(\bc\)-weight, and scaling a rational point of a region meeting the open half-plane \(\bc\bz<0\) gives a lattice point of negative \(\bc\)-weight.
	The region labels in \(\cN_{\bv}\) therefore form a consecutive block \(T_1,\ldots,T_m\) in the open semicircle of directions with \(\bc\bz>0\); this semicircle crosses each of the \(p\) central lines once, its two end regions are cut by the line \(\bc\bz=0\), and a central ray contributes no additional support, so \(m=p-1\).

	A nonzero point of \(\cK_0(\bv,B)\) lies on at most one line of the arrangement, and perturbing the point off that line when necessary gives an interior point; hence \(\cK_0(\bv,B)\) is either \(\{\bzero\}\) or a cone with nonempty interior.
	If \(\cK_0(\bv,B)\) has nonempty interior, then \(\cK_0(\bv,B)\) is contained in the closed half-plane \(\bc\bz\geq0\), because a closed convex cone with nonempty interior is the closure of that interior, so \(I_{\bzero}\) is one of \(T_1,\ldots,T_m\).
	If \(\cK_0(\bv,B)=\{\bzero\}\), then \(I_{\bzero}\) is realized only at the origin and is an additional member of \(\cN_{\bv}\).
	This proves the asserted formula for \(\#\cN_{\bv}\).

	Let \(T_0\) and \(T_{m+1}\) be the region labels adjacent to the two ends of this block; both lie outside \(\cN_{\bv}\).
	Every index in \(F\) belongs to \(K_{\cN_{\bv}}\), so \(E\subseteq J\), and each index in \(J\) changes its membership in the region labels at most once in passing from \(T_0\) to \(T_{m+1}\).
	Hence, for \(j\in E\), the labels among \(T_1,\ldots,T_m\) omitting \(j\) form a nonempty initial segment, a nonempty terminal segment, or all of \(T_1,\ldots,T_m\), and in the last case exactly one of \(T_0\) and \(T_{m+1}\) omits \(j\) because the two corresponding regions are antipodal.
	Call \(j\) initial if the labels omitting \(j\) form a proper initial segment, or if all of \(T_1,\ldots,T_m\) omit \(j\) and \(T_0\) omits \(j\), and terminal otherwise; then an initial index is omitted by both \(T_0\) and \(T_1\), and a terminal index is omitted by both \(T_m\) and \(T_{m+1}\).
	Ordering the initial indices first and the terminal indices after them gives an ordering \(j_1,\ldots,j_c\) of \(E\).

	Lemma~\ref{lem:central-endpoint-union} shows that \(T_0\cup T_1\) is contained in every union \(I\cup J'\) in which \(I\in\cN_{\bv}\) and \(J'\in\sS(\bv)\setminus\cN_{\bv}\) both omit an initial index \(j_a\), and that \(T_m\cup T_{m+1}\) has the same property for a terminal index.
	The pairs \((T_1,T_0)\) and \((T_m,T_{m+1})\) realize these two unions, and, when \(I_{\bzero}\) is an additional member of \(\cN_{\bv}\), the set \(I_{\bzero}\) is realized only at the origin, contains \(j_a\), and hence does not occur among these pairs.
	Put \(Y_{a-1}=\{j_1,\ldots,j_{a-1}\}\), and note that taking the colon by the product of the variables indexed by \(Y_{a-1}\) deletes those indices from every squarefree generator and preserves the inclusions supplied by Lemma~\ref{lem:central-endpoint-union}.
	Therefore \(P^{(a-1)}+\langle t_{j_a}\rangle = \left\langle t_{j_a},\ \bt^{(T_0\cup T_1)\setminus(K_{\cN_{\bv}}\cup Y_{a-1})} \right\rangle\) for an initial index, and the same formula with \(T_0\cup T_1\) replaced by \(T_m\cup T_{m+1}\) holds for a terminal index.
	If the second monomial is \(1\), this ideal is the unit ideal; otherwise the second monomial is the unique minimal generator not divisible by \(t_{j_a}\), so the ideal has the form \eqref{eq:principal-boundary-slice}.
\end{proof}

We call a fake exponent \(\bv\) binary if \(v_j\in\{-1,0\}\) for every \(j\).
We call \(\bv\) moving-binary if \(v_j\in\{-1,0\}\) for \(j\in J_{\mov}(\bv)\).

\begin{theorem}
	\label{thm:rank-two-moving-binary-vanishing}
	Suppose that \(\rank L=2\), that \(\bv\) is a moving-binary fake exponent for a generic direction satisfying \eqref{eq:moving-direction-genericity}, and that the rows of \(B\) indexed by \(J_{\mov}(\bv)\) are pairwise nonproportional.
	Then, for the distinguished collection, \(\fD_{\cN_{\bv}}(e)=0\).
	Moreover, if \(p=\#J_{\mov}(\bv)\), then
	\[
		\#\cN_{\bv}
		=
		\begin{cases}
			p-1, & \cK_0(\bv,B)\text{ has nonempty interior}, \\
			p,   & \cK_0(\bv,B)=\{\bzero\}.
		\end{cases}
	\]
	If the distinguished collection is ordered, then its ambient and intrinsic coefficient spaces are equal.
\end{theorem}

\begin{proof}
	Lemma~\ref{lem:planar-principal-flag} supplies an ordering satisfying condition \eqref{eq:principal-boundary-slice}.
	Every index occurring in a generator of one of the corresponding ideals \(\fj_a\) belongs to \(J_{\mov}(\bv)\).
	If \(j\) and \(k\) are two distinct such indices, the rows \(\bb_j\) and \(\bb_k\) are linearly independent.
	Gale duality gives \(\rank A_{\{1,\ldots,n\}\setminus\{j,k\}}=d\).
	Corollary~\ref{cor:principal-boundary-flag} therefore gives \(\fD_{\cN_{\bv}}(e)=0\).
\end{proof}

\begin{remark}
	The row matroid of a Gale dual is the dual of the column matroid of \(A\); see \cite[Chapter~2]{Oxl11}.
	The cosimplicity of the column matroid implies the hypothesis on the rows in Theorem~\ref{thm:rank-two-moving-binary-vanishing}.
	The theorem requires this hypothesis only for the rows indexed by \(J_{\mov}(\bv)\).
	Rows of \(B\) whose indices lie outside \(J_{\Z}(\bv)\) can be proportional, and an index with a zero row of \(B\) either belongs to every member of \(\sS(\bv)\) or belongs to none.
\end{remark}

\begin{remark}
	\label{rem:rank-two-decision}
	In lattice rank 2, the finite presentation is effective for one fake exponent \(\bv\), under an exact input model in which \(A\), \(B\), \(J_{\Z}(\bv)\), and the integers \(v_j\), \(j\in J_{\Z}(\bv)\), are given exactly and each entry of \(\bc=\bw B\) is a real algebraic number represented by a defining polynomial in \(\Z[x]\) and a rational isolating interval.
	The regions of the arrangement in \eqref{eq:rank-two-affine-arrangement} are enumerated by incremental insertion of the lines, and Lemma~\ref{lem:rank-two-affine-support-bound} bounds their number, so the subsets of \(J_{\mov}(\bv)\) need not be enumerated.
	A region label \(I\) is realized exactly when an integral point lies in the polyhedron cut out by \(\bb_j\bz\leq-v_j-1\) for \(j\in I\) and \(\bb_j\bz\geq-v_j\) for \(j\in J_{\Z}(\bv)\setminus I\).
	Whether a polyhedron contains an integral point is decidable \cite[Chapters~16 and~18]{Sch86}, and each fiber is Presburger-definable, so Theorem~1.3 of \cite{GS66} gives an effectively computable semilinear description of the fiber.
	Proposition~3 of \cite{CH16} gives explicit bounds for a semilinear decomposition of the integral solutions of a system of linear inequalities.
	Membership in \(\cN_{\bv}\) is then given by a finite list of sign conditions on the generators of the semilinear pieces, decided by exact real-algebraic arithmetic \cite[Chapter~2]{BPR06}, so \(\Gamma_{\cN_{\bv}}\) is computable with entries in \(\Q[s_1,s_2]\).
	Gr\"obner bases compute \(\ker\Gamma_{\cN_{\bv}}\) and the Hilbert series of \(S_0/J_{\cN_{\bv}}^{\amb}\) and \(S_0/J_{\cN_{\bv}}^{\intr}\) \cite[Chapter~5, Sections~2--3]{CLO05}, \cite[Chapters~2 and~9]{CLO15}.
	Theorem~\ref{thm:finite-boundary-classification} gives the Hilbert series of \(\fD_{\cN_{\bv}}(e)\) as the first series minus the second and the dimension of \(\fD_{\cN_{\bv}}(e)\) as the value of this difference at \(u=1\).
	This procedure treats one fake exponent; it does not enumerate isomorphism classes of rank 2 obstruction modules and gives no complexity bound.
\end{remark}

The negative support of a fake exponent does not determine the obstruction, as the following proposition shows.

\begin{proposition}
	\label{prop:equal-support-different-obstruction}
	The negative support of a fake exponent does not determine the rank 2 obstruction, even when a proportional pair of integral rows of \(B\) is fixed.
	More precisely, let
	\begin{equation}
		A=
		\begin{bmatrix}
			1  & 1  & 1 & 0 & 0 \\
			-1 & -2 & 0 & 1 & 0 \\
			1  & 2  & 0 & 0 & 1
		\end{bmatrix},
		\qquad
		B=
		\begin{bmatrix}
			1  & 2  \\
			0  & -1 \\
			-1 & -1 \\
			1  & 0  \\
			-1 & 0
		\end{bmatrix}.
		\label{eq:equal-support-AB}
	\end{equation}
	Put
	\[
		\varepsilon=\frac{\sqrt2}{100},
		\qquad
		\bw=(5,3-\varepsilon,0,0,0),
	\]
	and take \(\bv=\tp{(0,0,-3,0,-2)}\) and \(\bv^\flat=\tp{(0,0,-1,0,-1)}\).
	Then both vectors are fake exponents for the same generic direction, and \(\nsupp(\bv^\flat)=\nsupp(\bv)\).
	For \(\widetilde{\bv}\in\{\bv,\bv^\flat\}\), put \(e_{\widetilde{\bv}} = \bt^{ \nsupp(\widetilde{\bv}) \setminus K_{\cN_{\widetilde{\bv}}} }\).
	Then
	\begin{equation}
		\fD_{\cN_{\bv}}(e_{\bv})=0, \qquad \Hilb\bigl(\fD_{\cN_{\bv^\flat}}(e_{\bv^\flat});u\bigr)=u.
		\label{eq:depth-sensitive-obstructions}
	\end{equation}
\end{proposition}

\begin{proof}
	We have \(AB=0\) and \((1,1,1)A=(1,1,1,1,1)\), and the minor of \(B\) on rows two and four is \(1\).
	Thus the columns of \(B\) form a \(\Z\)-basis of \(L\), and \(A\) is homogeneous.
	Moreover, \(\bw B=(5,7+\varepsilon)\), so the direction is generic and \(\bw B\) is not proportional to any row of \(B\).
	A Buchberger calculation gives the reduced toric Gr\"obner basis with initial monomials \(\partial_1^2\), \(\partial_1\partial_4\), \(\partial_1\partial_5\), and \(\partial_2\partial_4^2\); see \cite[Chapter~2]{CLO15} for the criteria used.
	Each of these monomials contains a variable indexed by \(1\), \(2\), or \(4\), and the corresponding coordinates of \(\bv\) and \(\bv^\flat\) vanish, so \(\bv\) and \(\bv^\flat\) are fake exponents.
	For compactness, write \(i_1\cdots i_k\) for \(\{i_1,\ldots,i_k\}\).
	By solving the five sign inequalities in \eqref{eq:rank-two-affine-sign-test} over the integers and comparing the weights of the resulting support fibers, we obtain \(\cN_{\bv}=\{35,235,2345\}\) and \(\cN_{\bv^\flat}=\{234,35,235\}\).
	The linear forms are \((\ell_1,\ldots,\ell_5)=(s_1+2s_2,-s_2,-s_1-s_2,s_1,-s_1)\).
	For \(\bv\), the two antichains and the set \(E_{\bv}\) are \(\cG_{\cN_{\bv}}=\{\varnothing\}\), \(\cH_{\cN_{\bv}}=\{24,1\}\), and \(E_{\bv}=\varnothing\), so Theorem~\ref{thm:finite-boundary-classification} gives \(J_{\cN_{\bv}}^{\amb}=J_{\cN_{\bv}}^{\intr}=\langle s_2^2,s_1+2s_2\rangle\) and hence \(\fD_{\cN_{\bv}}(e_{\bv})=0\).
	For \(\bv^\flat\), the corresponding data are \(\cG_{\cN_{\bv^\flat}}=\{24,5\}\), \(\cH_{\cN_{\bv^\flat}}=\{24,15\}\), and \(E_{\bv^\flat}=5\), and the matrix \(\Gamma_{\cN_{\bv^\flat}}\) in Theorem~\ref{thm:finite-boundary-classification} gives \(J_{\cN_{\bv^\flat}}^{\amb}=\langle s_1+2s_2,s_2^2\rangle\) and \(J_{\cN_{\bv^\flat}}^{\intr}=\langle s_1,s_2\rangle\).
	The quotient of these two ideals is generated in degree one by the class of \(s_2\) and has square zero, so the second identity in \eqref{eq:depth-sensitive-obstructions} follows.
	Finally, the fourth and fifth rows of \(B\) are nonzero and proportional.
\end{proof}

Theorem~\ref{thm:rank-two-moving-binary-vanishing} yields the following corollary.

\begin{corollary}
	\label{cor:cosimple-binary-rank-two}
	Suppose that \(\rank L=2\), that \(\bv\) is a fake exponent for a generic direction \(\bw\), that the column matroid of \(A\) is cosimple, that \(\bv\) is moving-binary, and that \(\bw B\) is not proportional to any row of \(B\).
	Then the obstruction for the distinguished collection vanishes.
\end{corollary}

\begin{proof}
	Since the column matroid of \(A\) is cosimple, the rows of a rank 2 Gale dual are nonzero and pairwise nonproportional.
	Theorem~\ref{thm:rank-two-moving-binary-vanishing} applies.
\end{proof}

Corollary~\ref{cor:cosimple-binary-rank-two} yields the following lower bound on the lattice rank.

\begin{corollary}
	\label{cor:cosimple-binary-rank-threshold}
	Suppose that \(\bv\) is a moving-binary fake exponent for a generic direction \(\bw\), that the direction is nonzero on every nonzero element of \(L\), and that the column matroid of \(A\) is cosimple.
	If the intrinsic-perturbation obstruction module for the distinguished collection is nonzero, then \(\rank L\geq3\).
	The three-connected example in Theorem~\ref{thm:three-connected-obstruction} attains this bound.
\end{corollary}

\begin{proof}
	If \(\bw B\) were proportional to a row of \(B\), then \(\bw B\) would vanish on a nonzero integral vector in the kernel of that row, contrary to the hypothesis on the direction.
	Hence the rank 2 case follows from Corollary~\ref{cor:cosimple-binary-rank-two}, and the rank 0 and rank 1 cases follow from Corollary~\ref{cor:minimal-lattice-codimension}.
	The column matroid in Theorem~\ref{thm:three-connected-obstruction} is three-connected and has six elements, and hence it is cosimple.
	The fake exponent in Theorem~\ref{thm:three-connected-obstruction} is binary.
	Item (6) of Theorem~\ref{thm:three-connected-obstruction} states that the direction there is nonzero on every nonzero element of \(L\).
\end{proof}

The hypotheses of Theorem~\ref{thm:rank-two-moving-binary-vanishing} cannot be replaced by the assumption that the distinguished collection is ordered.

\begin{theorem}
	\label{thm:ordered-distinguished-counterexample}
	An ordered distinguished collection can have a nonzero obstruction.
	More precisely, let
	\begin{equation}
		A=
		\begin{bmatrix}
			20  & -5 & 3  & 0 & 0 \\
			-3  & 2  & -1 & 1 & 0 \\
			-16 & 4  & -1 & 0 & 1
		\end{bmatrix},
		\qquad
		B=
		\begin{bmatrix}
			0  & -1 \\
			3  & -4 \\
			5  & 0  \\
			-1 & 5  \\
			-7 & 0
		\end{bmatrix}.
		\label{eq:ordered-distinguished-AB}
	\end{equation}
	Put
	\[
		\vartheta=15+\frac{\sqrt2}{100},
		\qquad
		\bw=(\vartheta-4,1,0,0,0),
	\]
	so that \(\bw B=(3,-\vartheta)\), and take \(\bv=\tp{(0,1,0,-2,-3)}\).
	Then \(A\) is homogeneous, the direction is generic, and \(\bv\) is a fake exponent for \(\bbeta=A\bv=\tp{(-5,0,1)}\).
	The distinguished collection is \(\cN_{\bv}=\bigl\{\{3,4\},\{4,5\}\bigr\}\), this collection is ordered, and
	\begin{equation}
		J_{\cN_{\bv}}^{\amb}=\langle s_2,s_1^2\rangle
		\subsetneq
		\langle s_1,s_2\rangle=J_{\cN_{\bv}}^{\intr}.
		\label{eq:ordered-distinguished-ideals}
	\end{equation}
	Hence \(\dim_{\C}\fD_{\cN_{\bv}}(e)=1\).
\end{theorem}

\begin{proof}
	We have \(AB=0\), and the minor of \(B\) on rows one and four is \(-1\), so the columns of \(B\) form a \(\Z\)-basis of \(L\).
	Since \((1,1,1)A=(1,1,1,1,1)\), the matrix \(A\) is homogeneous.
	A Buchberger calculation gives the reduced toric Gr\"obner basis with initial monomials \(\partial_1\partial_2^4\), \(\partial_2^3\partial_3^5\), \(\partial_1\partial_2\partial_5^7\), and \(\partial_1\partial_5^{14}\); the corresponding exponent differences have \(\bw\)-weights \(\vartheta\), \(3\), \(\vartheta-3\), and \(\vartheta-6\).
	These weights are positive.
	The irrationality of \(\vartheta\) makes this direction generic, because \(3z_1-\vartheta z_2\) vanishes for \(\bz\in\Z^2\) only at the origin, so no nonzero element of \(L\) has zero \(\bw\)-weight.
	Each initial monomial is divisible by \(\partial_1\) or by \(\partial_2^3\), and \(\bv\) has \(v_1=0\) and \(v_2=1\), so \(\bv\) annihilates the distraction of \(\inw(I_A)\); together with \(A\bv=\bbeta\), this shows that \(\bv\) is a fake exponent.
	Writing \(\bz=\tp{(p_{\mathrm{lat}},q_{\mathrm{lat}})}\), we have \(\bv+B\bz=\tp{(-q_{\mathrm{lat}},\ 1+3p_{\mathrm{lat}}-4q_{\mathrm{lat}},\ 5p_{\mathrm{lat}},\ -2-p_{\mathrm{lat}}+5q_{\mathrm{lat}},\ -3-7p_{\mathrm{lat}})}\) and \(\bw B\bz=3p_{\mathrm{lat}}-\vartheta q_{\mathrm{lat}}\), so the five sign conditions read \(q_{\mathrm{lat}}\geq1\), \(3p_{\mathrm{lat}}-4q_{\mathrm{lat}}\leq-2\), \(p_{\mathrm{lat}}\leq-1\), \(p_{\mathrm{lat}}-5q_{\mathrm{lat}}\geq-1\), and \(p_{\mathrm{lat}}\geq0\).
	Solving them over the integers gives \(\sS(\bv)=\bigl\{ \{2,3\},\{1,2,3\},\{3,4\},\{2,3,4\},\{1,5\},\{1,2,5\},\{4,5\},\{1,4,5\} \bigr\}\), and comparing the weights of the resulting support fibers leaves \(\cN_{\bv}=\bigl\{\{3,4\},\{4,5\}\bigr\}\).
	No member of \(\sS(\bv)\) other than \(\{3,4\}\) is contained in \(\{3,4\}\), and none other than \(\{4,5\}\) is contained in \(\{4,5\}\), so \(\cN_{\bv}\) is ordered.
	Here \(K_{\cN_{\bv}}=\{4\}\) and \(e=t_5\), the two antichains are \(\cG_{\cN_{\bv}}=\bigl\{\{3\},\{5\}\bigr\}\) and \(\cH_{\cN_{\bv}}=\bigl\{\{2,3\},\{1,5\}\bigr\}\), and \(E=\{5\}\).
	The linear forms are \((\ell_1,\ldots,\ell_5)=(-s_2,3s_1-4s_2,5s_1,-s_1+5s_2,-7s_1)\), so \(J_{\cN_{\bv}}^{\intr} =\bigl\langle 5s_1(3s_1-4s_2),\ 7s_1s_2\bigr\rangle:(-7s_1) =\langle s_1,s_2\rangle\), and the matrix \(\Gamma_{\cN_{\bv}}\) in Theorem~\ref{thm:finite-boundary-classification} gives \(J_{\cN_{\bv}}^{\amb}=\langle s_2,s_1^2\rangle\).
	The quotient in \eqref{eq:ordered-distinguished-ideals} is spanned by the class of \(s_1\), so the obstruction has dimension one.
\end{proof}

\begin{remark}
	\label{rem:ordered-distinguished-scope}
	The example in Theorem~\ref{thm:ordered-distinguished-counterexample} fails two of the three hypotheses of Theorem~\ref{thm:rank-two-moving-binary-vanishing}.
	The fake exponent is not moving-binary, because \(v_4=-2\) and \(v_5=-3\).
	The rows \(\bb_3=(5,0)\) and \(\bb_5=(-7,0)\) of \(B\) are proportional, so the column matroid is not cosimple.
	The remaining hypothesis holds: since \(\vartheta\) is irrational, \(\bw B\) is not proportional to any row of \(B\), so \eqref{eq:moving-direction-genericity} is satisfied.
	The example therefore shows that the assumption that the collection is ordered does not by itself force the obstruction to vanish, and it does not show that either of the two failing hypotheses is individually necessary.
\end{remark}

Since the distinguished collection here is ordered, its intrinsic and ambient coefficient spaces are defined and differ.
This settles the case of an ordered distinguished collection in Question~7.3 of \cite{OS25}.

\section{Normality of \texorpdfstring{\(\cC(\bw)\)}{C(w)} does not eliminate the obstruction}
\label{sec:normal-counterexamples}

In this section, we give two counterexamples for which \(\cC(\bw)\) is normal.

\subsection{A counterexample of lattice rank 2}

The following theorem gives a counterexample of lattice rank 2.

\begin{theorem}
	\label{thm:lattice-counterexample}
	There exist a homogeneous \(A\)-hypergeometric system, a generic direction, a fake exponent \(\bv\), and an ordered negative support family \(\cN^{\circ}\subseteq\cN_{\bv}\) for which \(\Phi(Q_{\cN^{\circ}}(\bt):e) \subsetneq \Pi_{\cN^{\circ}}:m_{\bv,\cN^{\circ}}\).
	The affine semigroup \(\cC(\bw)\) is normal, and \(\dim_{\C}\fD_{\cN^{\circ}}(e)=1\).
	The distinguished collection \(\cN_{\bv}\) is not ordered, but it gives the same two monomial ideals, the same monomial \(e\), and hence the same obstruction module as the ordered family \(\cN^{\circ}\).
	For the ordered family \(\cN^{\circ}\), ambient perturbation supplies a logarithmic coefficient that cannot be obtained by intrinsic perturbation.
\end{theorem}

\begin{proof}
	Let
	\begin{equation}
		A=
		\begin{bmatrix}
			1 & 1 & 1 & 1 & 1 \\
			2 & 0 & 0 & 3 & 3 \\
			3 & 3 & 2 & 0 & 3
		\end{bmatrix},
		\qquad
		B=
		\begin{bmatrix}
			3  & 0  \\
			2  & 3  \\
			-3 & -3 \\
			1  & 1  \\
			-3 & -1
		\end{bmatrix}.
		\label{eq:lattice-counterexample-AB}
	\end{equation}
	We have \(AB=0\), and the minor of \(B\) on rows two and four is \(-1\), so the columns of \(B\) form a \(\Z\)-basis of \(L\).
	The first row of \(A\) also shows that the configuration is homogeneous.
	Fix \(\vartheta=40+\sqrt2\) and \(\bw=(-(2\vartheta+3)/9,\ \vartheta/3,\ 0,0,0)\), so that \(\bw B=(-1,\vartheta)\); the irrationality of \(\vartheta\) makes this direction generic.
	The reduced toric Gr\"obner basis is
	\begin{equation}
		\begin{aligned}
			f_1&=\underline{\partial_3^3\partial_5^3}
			-\partial_1^3\partial_2^2\partial_4,\\
			f_2&=\underline{\partial_2^3\partial_4}
			-\partial_3^3\partial_5,\\
			f_3&=\underline{\partial_2\partial_5^2}-\partial_1^3.
		\end{aligned}
		\label{eq:lattice-counterexample-GB}
	\end{equation}
	The first two initial monomials are relatively prime.
	The remaining two \(S\)-polynomials satisfy \(S(f_1,f_3)=-\partial_1^3f_2\) and \(S(f_2,f_3)=-f_1\), so Buchberger's criterion applies \cite[Chapter~2]{CLO15}.
	To see that the ideal generated by \(f_1,f_2,f_3\) is all of \(I_A\), write the exponent difference of two standard monomials of the same \(A\)-degree as \(B\tp{(p_{\mathrm{lat}},q_{\mathrm{lat}})}\).
	In each of the cases \(p_{\mathrm{lat}}+q_{\mathrm{lat}}>0\), \(p_{\mathrm{lat}}+q_{\mathrm{lat}}<0\), and \(p_{\mathrm{lat}}+q_{\mathrm{lat}}=0\), the three displayed initial monomials force \(p_{\mathrm{lat}}=q_{\mathrm{lat}}=0\), so distinct standard monomials cannot have the same \(A\)-degree.
	Thus every lattice binomial reduces to zero, which proves that \(f_1,f_2,f_3\) generate \(I_A\).
	The three binomials in \eqref{eq:lattice-counterexample-GB} have lattice coordinates \((-1,0)\), \((0,1)\), and \((-1,1)\), and hence \(\cC(\bw)=\N(-1,0)+\N(0,1)=\{(p_{\mathrm{lat}},q_{\mathrm{lat}})\in\Z^2\mid p_{\mathrm{lat}}\leq0,\ q_{\mathrm{lat}}\geq0\}\).
	This affine semigroup is saturated and therefore normal.
	Take \(\bv=\tp{(-1,0,-1,0,0)}\) and \(\bbeta=A\bv=\tp{(-2,-2,-5)}\).
	Every initial monomial in \eqref{eq:lattice-counterexample-GB} contains \(\partial_2\) or \(\partial_5\), so the distraction of that monomial vanishes at \(\bv\), and \(\bv\) is a fake exponent.
	If we write a lattice point as \(B\tp{(p_{\mathrm{lat}},q_{\mathrm{lat}})}\), the negative support tests are
	\begin{equation}
		\begin{array}{c|c}
			1\in I_{(p_{\mathrm{lat}},q_{\mathrm{lat}})}&p_{\mathrm{lat}}\leq0\\
			2\in I_{(p_{\mathrm{lat}},q_{\mathrm{lat}})}&2p_{\mathrm{lat}}+3q_{\mathrm{lat}}\leq-1\\
			3\in I_{(p_{\mathrm{lat}},q_{\mathrm{lat}})}&p_{\mathrm{lat}}+q_{\mathrm{lat}}\geq0\\
			4\in I_{(p_{\mathrm{lat}},q_{\mathrm{lat}})}&p_{\mathrm{lat}}+q_{\mathrm{lat}}\leq-1\\
			5\in I_{(p_{\mathrm{lat}},q_{\mathrm{lat}})}&3p_{\mathrm{lat}}+q_{\mathrm{lat}}\geq1.
		\end{array}
		\label{eq:lattice-counterexample-signs}
	\end{equation}
	A direct case distinction using these five integral inequalities gives exactly the following eight supports.
	\begin{equation}
		\begin{array}{c|c|c}
			I&(p_{\mathrm{lat}},q_{\mathrm{lat}})&\text{semigroup containment or negative weight}\\ \hline
			\{1,3\}&(0,0)&\cF_I\subseteq\cC(\bw)\\
			\{1,4\}&(-3,2)&\cF_I\subseteq\cC(\bw)\\
			\{1,3,5\}&(0,1)&\cF_I\subseteq\cC(\bw)\\ \hline
			\{2,4\}&(1,-3)&-1-3\vartheta<0\\
			\{1,2,4\}&(0,-1)&-\vartheta<0\\
			\{3,5\}&(1,0)&-1<0\\
			\{2,3,5\}&(1,-1)&-1-\vartheta<0\\
			\{2,4,5\}&(1,-2)&-1-2\vartheta<0.
		\end{array}
		\label{eq:lattice-counterexample-table}
	\end{equation}
	The inequalities in \eqref{eq:lattice-counterexample-signs} also show that the first three fibers lie in \(\cC(\bw)\).
	In every other fiber, the displayed point has negative weight.
	Consequently,
	\begin{equation}
		\cN_{\bv}
		=
		\{\{1,3\},\{1,4\},\{1,3,5\}\},
		\qquad
		K_{\cN_{\bv}}=\{1\}.
		\label{eq:lattice-counterexample-N}
	\end{equation}
	The realized negative support \(\{3,5\}\) is properly contained in \(\{1,3,5\}\), so \(\cN_{\bv}\) is not ordered.
	Set \(\cN^{\circ}=\{\{1,3\},\{1,4\}\}\).
	The same support table shows that \(\cN^{\circ}\) is ordered and that the two families give the same two monomial ideals and the same monomial \(e\):
	\begin{equation}
		M_{\cN^{\circ}}=M_{\cN_{\bv}}=\langle t_3,t_4\rangle,
		\qquad
		P_{\cN^{\circ}}(\bt)=P_{\cN_{\bv}}(\bt)=\langle t_2t_4,t_3t_5\rangle,
		\qquad
		e=t_3.
		\label{eq:lattice-counterexample-MP}
	\end{equation}
	Write \(\cN=\cN^{\circ}\), and let \((\ell_1,\ldots,\ell_5)=(3s_1,2s_1+3s_2,-3s_1-3s_2,s_1+s_2,-3s_1-s_2)\).
	The linear forms satisfy \(\ell_4=(2\ell_2-\ell_5)/7\).
	The two antichains and the set \(E\) are \(\cG_{\cN}=\{\{3\},\{4\}\}\), \(\cH_{\cN}=\{\{2,4\},\{3,5\}\}\), and \(E=\{3\}\).
	Corollary~\ref{cor:two-support-principal-formula} therefore gives
	\begin{equation}
		J_{\cN}^{\amb}=\langle\ell_5,\ell_2\ell_4\rangle=\langle3s_1+s_2,(s_1+s_2)^2\rangle, \qquad J_{\cN}^{\intr}=\langle\ell_2,\ell_5\rangle=\langle s_1,s_2\rangle.
		\label{eq:lattice-counterexample-extended-ideal}
	\end{equation}
	Equivalently, we take inverse images under \(\Phi\) and reduce the linear generators of \(U=\langle A\bt\rangle\) to obtain
	\begin{align}
		(UM_{\cN}+P_{\cN}(\bt)):t_3
		&=
		\langle2t_1+3t_4,2t_2-7t_4,t_3+3t_4,t_5,t_4^2\rangle,
		\label{eq:lattice-counterexample-left-colon}\\
		(U+P_{\cN}(\bt)):t_3
		&=
		\langle t_1,t_2,t_3,t_4,t_5\rangle.
	\end{align}
	Theorem~\ref{thm:finite-boundary-classification} now yields \(\fD_{\cN}(e)\simeq\langle s_1,s_2\rangle/\langle3s_1+s_2,(s_1+s_2)^2\rangle=\C\,\overline{s_1+s_2}\), of dimension one.
	The inverse systems of \(J_{\cN}^{\intr}\) and \(J_{\cN}^{\amb}\) are \(\C\) and \(\Span\{1,\partial_{s_1}-3\partial_{s_2}\}\), respectively, so the additional coefficient supplied by ambient perturbation has the leading logarithmic term \((\bb^{(1)}-3\bb^{(2)})\cdot\log\bx=3\log x_1-7\log x_2+6\log x_3-2\log x_4\).
	Theorem~5.5 and Proposition~6.2 of \cite{OS25} identify this coefficient as an element of the ambient coefficient space that does not lie in the intrinsic coefficient space.
\end{proof}

\begin{proposition}
	\label{prop:lattice-counterexample-koszul-class}
	In the notation of Theorem~\ref{thm:lattice-counterexample}, let \(\theta_1,\theta_2,\theta_3\) be the entries of \(A\bt\), and recall that \(R_0=\C[t_1,\ldots,t_5]\).
	The element
	\[
		z=
		\left(-2t_1,0,\frac{2}{3}t_1\right)
		\in
		K_1(\theta_1,\theta_2,\theta_3;R_0/M_{\cN}^0)
	\]
	is a cycle whose image under the connecting map used in \eqref{eq:boundary-tor-cokernel} is
	\[
		7\overline{t_3t_4}
		\in
		\frac{(U_0\cap M_{\cN}^0)+P_{\cN}^0(\bt)}
		{U_0M_{\cN}^0+P_{\cN}^0(\bt)}.
	\]
	This class is nonzero.
	The class of \(t_4\) in \(\Tor^{R_0}_1 \left( R_0/\langle t_3\rangle,R_0/(U_0+P_{\cN}^0(\bt)) \right)\) maps to a generator of the cokernel in \eqref{eq:polynomial-tor-cokernel}.
	The class of \(t_4\) maps under \(\Phi_0\) to \(\overline{s_1+s_2}\in(J_{\cN}^{\intr}/J_{\cN}^{\amb})_1\).
	The operator \(\partial_{s_1}-3\partial_{s_2}\) represents a class in \(V_{\cN,1}^{\amb}/V_{\cN,1}^{\intr}\) that pairs nontrivially with \(\overline{s_1+s_2}\).
\end{proposition}

\begin{proof}
	The element \(z\) has a lift \(\widetilde z=\left(-2t_1+6t_4,\frac{1}{3}t_3-2t_4,\frac{2}{3}t_1\right)\in K_1(\theta_1,\theta_2,\theta_3;R_0)\), and a direct calculation gives
	\begin{equation}
		\theta_1(-2t_1+6t_4)+\theta_2\left(\frac{1}{3}t_3-2t_4\right)+\theta_3\left(\frac{2}{3}t_1\right)=7t_3t_4+6t_2t_4+t_3t_5.
		\label{eq:lattice-counterexample-koszul-boundary}
	\end{equation}
	The right-hand side belongs to \(M_{\cN}^0\), so \(z\) is a cycle; the right-hand side differs from \(7t_3t_4\) by \(6t_2t_4+t_3t_5\), which belongs to \(P_{\cN}^0(\bt)=\langle t_2t_4,t_3t_5\rangle\), and this containment gives the stated image under the connecting map.
	Equation \eqref{eq:lattice-counterexample-left-colon} and the relation \(\Phi_0(t_4)=s_1+s_2\notin J_{\cN}^{\amb}\) show that \(t_3t_4\notin U_0M_{\cN}^0+P_{\cN}^0(\bt)\), so the class is nonzero.
	Equation \eqref{eq:lattice-counterexample-koszul-boundary} also gives \(t_3t_4\in U_0+P_{\cN}^0(\bt)\), so \(t_4\) represents the stated \(\Tor_1\) class.
	This class does not come from the first term in \eqref{eq:polynomial-tor-cokernel}, because otherwise we would have \(t_4\in Q_{\cN}^0(\bt):t_3\), contrary to \eqref{eq:lattice-counterexample-left-colon}.
	Since the cokernel is one-dimensional, this class generates the cokernel.
	Finally, \((\partial_{s_1}-3\partial_{s_2})(s_1+s_2)=-2\), and the last assertion follows from \eqref{eq:tor-coefficient-exact-sequence}.
\end{proof}

\begin{remark}
	The strict inclusion in Theorem~\ref{thm:lattice-counterexample} answers the question of nonequality in \cite[Proposition~6.2 and Question~7.3]{OS25} for the fixed fake exponent \(\bv\) and the ordered family \(\cN^{\circ}\).
	The distinguished collection \(\cN_{\bv}\) is not ordered.
	We obtain the same intrinsic-perturbation obstruction module from \(\cN_{\bv}\) as from \(\cN^{\circ}\), without using the coefficient spaces of \(\cN_{\bv}\).
	Theorem~\ref{thm:ordered-distinguished-counterexample} answers the same question for an ordered distinguished collection.
\end{remark}

\begin{remark}
	For the ideal \(M_{\cN}\) in \eqref{eq:lattice-counterexample-MP}, the complex \(\Delta_{\cN}\) is the simplex with facet \(F=\{1,2,5\}\).
	Thus the Stanley--Reisner ring of \(\Delta_{\cN}\) is \(\C[t_1,t_2,t_5]\), which is Cohen--Macaulay.
	On the other hand,
	\[
		A_F=
		\begin{bmatrix}
			1 & 1 & 1 \\
			2 & 0 & 3 \\
			3 & 3 & 3
		\end{bmatrix},
		\qquad
		\rank(A_F)=2<3.
	\]
	Hence the linear forms \(\theta_1,\ldots,\theta_d\) do not form a system of parameters for the Stanley--Reisner ring \(\C[\Delta_{\cN}]\).
	The nonzero class \(\overline{s_1+s_2}\in J_{\cN}^{\intr}/J_{\cN}^{\amb}\) computed above shows that the conclusion of Theorem~\ref{thm:CM-regular-sequence} fails here.
\end{remark}

\begin{remark}
	The fake exponent in Theorem~\ref{thm:lattice-counterexample} is binary, and the direction is nonzero on every nonzero element of \(L\).
	Condition \eqref{eq:principal-boundary-slice} holds with a single step:
	\[
		P_{\cN}^0(\bt)+\langle t_3\rangle=\langle t_3,t_2t_4\rangle.
	\]
	The Stanley--Reisner complex of \(\langle t_3,t_2t_4\rangle\) has facets \(\{1,4,5\}\) and \(\{1,2,5\}\).
	The submatrix of \(A\) on the facet \(\{1,2,5\}\) has rank 2, as displayed above.
	Equivalently, the third and fourth rows of \(B\) in \eqref{eq:lattice-counterexample-AB} are proportional.
	Since \(\bw B=(-1,\vartheta)\) with \(\vartheta\) irrational, and since every row of \(B\) is rational and nonzero, condition \eqref{eq:moving-direction-genericity} holds.
	Thus every hypothesis of Theorem~\ref{thm:rank-two-moving-binary-vanishing} holds except the nonproportionality of the third and fourth rows of \(B\), and this failure is the reason that the rank condition in Theorem~\ref{thm:factorwise-regularity} does not hold here.
\end{remark}

\subsection{A three-connected counterexample}

The following theorem shows that three-connectivity of the column matroid does not force the obstruction to vanish.

\begin{theorem}
	\label{thm:three-connected-obstruction}
	There exist a homogeneous \(A\)-hypergeometric system of lattice rank 3, a generic direction, a fake exponent \(\bv\), and an ordered negative support family \(\cN^{\circ}\subseteq\cN_{\bv}\) such that the following properties hold.
	\begin{enumerate}
		\item The column matroid of \(A\) is three-connected.
		\item The affine semigroup \(\cC(\bw)\) is generated by a \(\Z\)-basis of \(L\).
		      In particular, this affine semigroup is normal and free, and its cone is unimodular.
		\item The Stanley--Reisner ring \(\C[\Delta_{\cN^{\circ}}]\) is a polynomial ring.
		\item The obstruction module satisfies \(\Hilb\bigl(\fD_{\cN^{\circ}}(e);u\bigr)=u\).
		\item The fake exponent \(\bv\) is binary.
		\item The direction is nonzero on every nonzero element of \(L\).
	\end{enumerate}
	The distinguished collection \(\cN_{\bv}\) is not ordered, but it gives the same two monomial ideals, the same monomial \(e\), and hence the same obstruction module as the ordered family \(\cN^{\circ}\).
	For the ordered family \(\cN^{\circ}\), ambient perturbation supplies a coefficient with the leading logarithmic term \(\log x_2-2\log x_3+\log x_6\), and intrinsic perturbation does not supply this coefficient at the fake exponent \(\bv\).
\end{theorem}

\begin{proof}
	Set
	\[
		A=
		\begin{bmatrix}
			1 & 1 & 1 & 1  & 1 & 1 \\
			2 & 4 & 2 & -5 & 1 & 0 \\
			1 & 3 & 2 & -3 & 0 & 1
		\end{bmatrix},
		\qquad
		B=
		\begin{bmatrix}
			3  & 0  & -3 \\
			2  & 3  & -1 \\
			-3 & -3 & 1  \\
			1  & 1  & -1 \\
			-3 & -1 & 3  \\
			0  & 0  & 1
		\end{bmatrix}.
	\]
	We have \(AB=0\), and the minor of \(B\) on rows two, four, and six is \(-1\), so the columns of \(B\) form a \(\Z\)-basis of \(L\).
	The first row of \(A\) also shows that the configuration is homogeneous.
	Put \(\varepsilon=1/1000\), \(\vartheta=41+\varepsilon\sqrt2\), and \(\varpi=3+\varepsilon\sqrt3\), and choose \(\bw=(-\varpi/3-\vartheta/9,\ 1-\varpi+2\vartheta/3,\ 1-\varpi+\vartheta/3,\ 0,0,0)\).
	Then \(\bw B=(-1,\vartheta,\varpi)\), whose entries are linearly independent over \(\Q\), so the direction is generic and is nonzero on every nonzero element of \(L\).
	Consider the eight binomials
	\[
		\begin{aligned}
			f_1&=\underline{\partial_2^2\partial_4}-\partial_3\partial_5\partial_6,&
			f_2&=\underline{\partial_3^3\partial_4}-\partial_5\partial_6^3,&
			f_3&=\underline{\partial_2\partial_3\partial_4}-\partial_5\partial_6^2,\\
			f_4&=\underline{\partial_3^2\partial_5^2}-\partial_1^3\partial_6,&
			f_5&=\underline{\partial_2\partial_5^2}-\partial_1^3,&
			f_6&=\underline{\partial_3\partial_5^3\partial_6}-\partial_1^3\partial_2\partial_4,\\
			f_7&=\underline{\partial_2\partial_6}-\partial_3^2,&
			f_8&=\underline{\partial_5^3\partial_6^2}-\partial_1^3\partial_3\partial_4.
		\end{aligned}
	\]
	The underlined monomials are the initial monomials for the weighted-degree order with weight \((1,279,156,51,51,51)\) and lexicographic tie-breaking, and also for the \(\bw\)-weight order.
	The \(S\)-polynomials of the pairs not covered by the product criterion reduce to zero, so Buchberger's criterion \cite[Chapter~2]{CLO15} shows that these binomials form a reduced Gr\"obner basis of the ideal \(J\) that they generate.
	No initial monomial is divisible by \(\partial_1\), and the division algorithm therefore gives \(J:\partial_1^\infty=J\).
	After we invert \(\partial_1\), the binomial \(f_5\) makes \(\partial_2\partial_5^2\) a unit, and each factor of a unit in a commutative ring is a unit, so \(f_5,f_7,f_1\) give \(\partial_2=\partial_1^3\partial_5^{-2}\), \(\partial_6=\partial_1^{-3}\partial_3^2\partial_5^2\), and \(\partial_4=\partial_1^{-9}\partial_3^3\partial_5^7\), after which the remaining five binomials vanish.
	The minor of \(A\) on columns one, three, and five is \(1\), so the toric monomials of these three columns are algebraically independent, whence \(J[\partial_1^{-1}]=I_A[\partial_1^{-1}]\) and \(J:\partial_1^{\infty}=J\) give \(J=I_A\).
	Among the lattice coordinates of the eight binomials are \(\bg_1=(-3,1,-3)\), \(\bg_2=(1,0,1)\), and \(\bg_3=(0,0,1)\).
	All eight lattice coordinates are nonnegative integral combinations of these three vectors, whose matrix has determinant of absolute value one, so \(\cC(\bw)=\N\bg_1+\N\bg_2+\N\bg_3=\{(p_{\mathrm{lat}},q_{\mathrm{lat}},r_{\mathrm{lat}})\in\Z^3\mid q_{\mathrm{lat}}\geq0,\ p_{\mathrm{lat}}+3q_{\mathrm{lat}}\geq0,\ r_{\mathrm{lat}}-p_{\mathrm{lat}}\geq0\}\).
	This affine semigroup is normal and free, and it is generated by a lattice basis.

	Take \(\bv=\tp{(-1,0,-1,0,0,0)}\) and \(\bbeta=A\bv=\tp{(-2,-4,-3)}\).
	Every initial monomial contains a variable at which \(\bv\) is zero, so \(\bv\) is a fake exponent, and the coordinates of \(\bv\) lie in \(\{-1,0\}\), so \(\bv\) is binary.
	For \(B\tp{(p_{\mathrm{lat}},q_{\mathrm{lat}},r_{\mathrm{lat}})}\), the index \(j\) belongs to \(I_{(p_{\mathrm{lat}},q_{\mathrm{lat}},r_{\mathrm{lat}})}\) exactly when \(p_{\mathrm{lat}}\leq r_{\mathrm{lat}}\), \(2p_{\mathrm{lat}}+3q_{\mathrm{lat}}-r_{\mathrm{lat}}\leq-1\), \(r_{\mathrm{lat}}\leq3p_{\mathrm{lat}}+3q_{\mathrm{lat}}\), \(p_{\mathrm{lat}}+q_{\mathrm{lat}}-r_{\mathrm{lat}}\leq-1\), \(-3p_{\mathrm{lat}}-q_{\mathrm{lat}}+3r_{\mathrm{lat}}\leq-1\), and \(r_{\mathrm{lat}}\leq-1\) for \(j=1,\ldots,6\), respectively.
	A direct case distinction using these six inequalities gives exactly twenty-eight realized negative supports, of which exactly the eight supports \(\{1,3\}\), \(\{1,4\}\), \(\{1,3,5\}\), \(\{1,6\}\), \(\{1,3,6\}\), \(\{1,4,6\}\), \(\{1,5,6\}\), and \(\{1,3,5,6\}\) index fibers contained in \(\cC(\bw)\); these eight sets form \(\cN_{\bv}\).
	For each of the remaining twenty fibers, direct substitution gives an integral point whose \((-1,41,3)\)-weight is at most \(-1\), and the perturbation to \((-1,\vartheta,\varpi)\) changes each of these weights by less than \(48/1000\), so every remaining fiber contains a negative-weight point.
	The realized negative support \(\{3,5\}\) is properly contained in \(\{1,3,5\}\), so \(\cN_{\bv}\) is not ordered.
	Set \(\cN^{\circ}=\{\{1,3\},\{1,4\},\{1,6\},\{1,3,6\},\{1,4,6\}\}\).
	The case distinction shows that \(\cN^{\circ}\) is ordered and that it gives the same two monomial ideals \(M=\langle t_3,t_4,t_6\rangle\) and \(P=\langle t_2t_4,t_3t_4,t_3t_5,t_2t_6,t_5t_6\rangle\) and the same monomial \(e=t_3\) as \(\cN_{\bv}\).

	Put \(U=\langle A\bt\rangle\), \(Q=UM+P\), and \(Q'=\langle t_6^2,t_1,t_2-t_6,t_3+2t_6,t_4,t_5\rangle\).
	Multiplying each generator of \(Q'\) by \(t_3\) and reducing by the three linear generators of \(U\) and by the generators of \(P\), we obtain \(Q'\subseteq Q:t_3\).
	To prove equality, let \(\varphi\) be the linear functional on quadratic forms that vanishes on every quadratic monomial except the nine monomials \(t_1t_4,t_1t_6,t_2t_3,t_3^2,t_3t_6,t_4^2,t_4t_5,t_4t_6,t_6^2\), on which it takes the values \(-3/4,-3/8,1,-2,1,-1/6,2/3,1/4,-7/8\), respectively.
	Direct substitution gives \(\varphi(Q_2)=0\) and \(\varphi(t_3t_6)=1\), so \(t_6\notin Q:t_3\), and \(1\notin Q:t_3\) by degree.
	Since \(\Rhat/Q'\simeq\C[t_6]/\langle t_6^2\rangle\), these exclusions prove \(Q:t_3=Q'\), and the same reductions show that every variable lies in \((U+P):t_3\), so \((U+P):t_3=\langle t_1,\ldots,t_6\rangle\).
	Applying \(\Phi\) to these two identities yields \(\Phi(Q:t_3)=\langle s_1-s_3,s_2,s_3^2\rangle\) and \(\Pi_{\cN^{\circ}}:m_{\bv,\cN^{\circ}}=\langle s_1,s_2,s_3\rangle\), whence \(\fD_{\cN^{\circ}}(e)\simeq\langle s_1,s_2,s_3\rangle/\langle s_1-s_3,s_2,s_3^2\rangle=\C\,\overline{s_3}\) with \(\deg\overline{s_3}=1\).
	The complex \(\Delta_{\cN^{\circ}}\) is the simplex on \(\{1,2,5\}\), so \(\C[\Delta_{\cN^{\circ}}]\) is the polynomial ring \(\C[t_1,t_2,t_5]\).
	Every pair of columns of \(A\) is independent, the only dependent triples are \(\{1,2,5\}\) and \(\{2,3,6\}\), and every four-column submatrix has rank 3, so the connectivity function of the column matroid has value at least two on every subset of size between two and four, which proves that this matroid is three-connected.
	Finally, the inverse systems of \(\langle s_1-s_3,s_2,s_3^2\rangle\) and \(\langle s_1,s_2,s_3\rangle\) are \(\Span\{1,\partial_{s_1}+\partial_{s_3}\}\) and \(\C\), respectively.
	Since \(\bb^{(1)}+\bb^{(3)}=\tp{(0,1,-2,0,0,1)}\), Theorem~5.5 and Proposition~6.2 of \cite{OS25} yield the logarithmic coefficient asserted in the statement.
\end{proof}

Theorem~\ref{thm:three-connected-obstruction} yields the following corollary to Theorem~\ref{thm:CM-regular-sequence}.

\begin{corollary}
	\label{cor:three-connectivity-not-enough}
	In the notation of Theorem~\ref{thm:three-connected-obstruction}, we can neither omit the hypothesis that \(\rank(A_F)=d\) on every facet from Theorem~\ref{thm:CM-regular-sequence} nor replace this hypothesis by three-connectivity of the column matroid, even when the affine semigroup \(\cC(\bw)\) is generated by a lattice basis and the Stanley--Reisner ring \(\C[\Delta_{\cN^{\circ}}]\) is a polynomial ring.
\end{corollary}

\begin{proof}
	The system in Theorem~\ref{thm:three-connected-obstruction} satisfies all the stated conditions and has a nonzero intrinsic-perturbation obstruction module.
	The unique facet \(F=\{1,2,5\}\) of \(\Delta_{\cN^{\circ}}\) satisfies \(\rank(A_F)=2<3\).
\end{proof}

\section{A family whose obstruction modules have unbounded dimension}
\label{sec:unbounded-families}

In this section, we construct a family of lattice rank 2 whose obstruction modules have unbounded dimension.

Write \(J_0=\{1,3\}\), \(J_1=\{1,4\}\), and \(J_2=\{1,3,5\}\) for the three members of the distinguished collection of the five-column example in \eqref{eq:lattice-counterexample-N}.

We prepare a lemma on images of Gr\"obner bases under a monomial homomorphism.

\begin{lemma}
	\label{lem:replicated-Buchberger}
	Let \(\rho:\C[\bx]\to\C[\bz]\) be a monomial homomorphism that is injective on monomials.
	Suppose that
	\begin{equation}
		\operatorname{lcm}(\rho(m),\rho(m'))
		=
		\rho(\operatorname{lcm}(m,m'))
		\label{eq:replicated-lcm}
	\end{equation}
	for all monomials \(m\) and \(m'\), and that monomial orders \(\prec\) on \(\C[\bx]\) and \(\prec'\) on \(\C[\bz]\) satisfy
	\[
		m\prec m' \quad\Longleftrightarrow\quad \rho(m)\prec'\rho(m').
	\]
	If \(G\) is a monic Gr\"obner basis with respect to \(\prec\), then \(\rho(G)\) is a Gr\"obner basis of the ideal generated by \(\rho(G)\) with respect to \(\prec'\).
	If \(G\) is reduced, then \(\rho(G)\) is reduced.
\end{lemma}

\begin{proof}
	The lcm identity and the compatibility of the two orders give \(S(\rho(g),\rho(h))=\rho(S(g,h))\), and every reduction by \(G\) maps to a reduction by \(\rho(G)\).
	Buchberger's criterion \cite[Chapter~2]{CLO15} proves the first assertion.
	The lcm identity and the injectivity of \(\rho\) also show that \(\rho\) reflects divisibility of monomials, and hence that \(\rho(G)\) is reduced.
\end{proof}

\subsection{A family of lattice rank 2}

The next construction repeats the rows of the matrix \(B\) in \eqref{eq:lattice-counterexample-AB}.
Thus the number of columns increases, and the lattice rank remains 2.
In \(\C[[s_1,s_2]]\), put \(\lambda_1=2s_1+3s_2\), \(\lambda_2=3s_1+s_2\), and \(\lambda_3=s_1+s_2\).

\begin{theorem}
	\label{thm:fixed-rank-unbounded}
	For every integer \(q\geq1\), there exist a homogeneous \(A\)-hypergeometric system of lattice rank 2 with \(5q\) columns, with a connected column matroid and with a normal affine semigroup generated by the reduced Gr\"obner basis vectors, a generic direction that is nonzero on every nonzero element of \(L\), a binary fake exponent \(\bv_q^\triangle\), and an ordered negative support family \(\cN_q^{\triangle,\circ}\), such that \(\Hilb\bigl( \fD_{\cN_q^{\triangle,\circ}}(e_q^\triangle);u \bigr)=u^q(1+u+\cdots+u^{q-1})^2\).
	In particular, \(\dim_{\C} \fD_{\cN_q^{\triangle,\circ}}(e_q^\triangle)=q^2\).
	The two ideals in \eqref{eq:defect-image} are
	\begin{equation}
		\Phi\bigl(Q_{\cN_q^{\triangle,\circ}}(\bt):e_q^\triangle\bigr)=\langle\lambda_2^q,\lambda_1^q\lambda_3^q\rangle, \qquad \Pi_{\cN_q^{\triangle,\circ}}:m_{\bv_q^\triangle,\cN_q^{\triangle,\circ}}=\langle\lambda_1^q,\lambda_2^q\rangle.
		\label{eq:fixed-rank-ideals}
	\end{equation}
	There is a graded isomorphism
	\[
		\fD_{\cN_q^{\triangle,\circ}}(e_q^\triangle)
		\simeq
		\frac{\C[s_1,s_2]}{\langle\lambda_2^q,\lambda_3^q\rangle}(-q).
	\]
\end{theorem}

\begin{proof}
	Let \(A_0\) and \(B_0\) be the matrices in \eqref{eq:lattice-counterexample-AB}, and let \(B_q^\triangle=\tp{(\tp{B_0}\ \cdots\ \tp{B_0})}\), with \(q\) copies of \(B_0\).
	The unimodular minor on rows two and four of any block shows that the columns of \(B_q^\triangle\) span a saturated lattice of rank 2, and since the all-ones row annihilates \(B_q^\triangle\), a basis of \(\ker_{\Z}(\tp{(B_q^\triangle)})\) beginning with this row gives a homogeneous matrix \(A_q^\triangle\) with \(\ker_{\Z}(A_q^\triangle)=\im_{\Z}(B_q^\triangle)\).
	Choose \(\bw_0^\triangle\) with \(\bw_0^\triangle B_0=(-1,\vartheta)\) for \(\vartheta=40+\sqrt2\), and set \(w_{ij}^\triangle=q^{-1}w_{0j}^\triangle\), so that \(\bw_q^\triangle B_q^\triangle=(-1,\vartheta)\) and the direction is generic and nonzero on every nonzero element of \(L\).
	Consider the monomial homomorphism \(\rho_q\) from \(\C[\partial_1,\ldots,\partial_5]\) to \(\C[\partial_{ij}\mid1\leq i\leq q,\ 1\leq j\leq5]\) sending \(\partial_j\) to \(\prod_{i=1}^q\partial_{ij}\).
	After removal of a common monomial factor, every lattice binomial for \(A_q^\triangle\) is the image under \(\rho_q\) of a lattice binomial for \(A_0\), because the exponent difference of a lattice binomial for \(A_q^\triangle\) repeats in all \(q\) blocks.
	Thus the images of the three binomials in \eqref{eq:lattice-counterexample-GB} generate \(I_{A_q^\triangle}\).
	The map \(\rho_q\) is injective on monomials and preserves least common multiples, and a compatible target monomial order is obtained by first comparing the exponent sums over the \(q\) blocks in the base monomial order and then breaking ties lexicographically.
	For a monomial \(m\), the exponent sum of \(\rho_q(m)\) over the \(q\) blocks is \(q\) times the exponent vector of \(m\), so the base and target monomial orders agree on the image of \(\rho_q\), and Lemma~\ref{lem:replicated-Buchberger} shows that the three images form the reduced toric Gr\"obner basis of \(I_{A_q^\triangle}\).
	Their lattice coordinates are \((-1,0)\), \((0,1)\), and \((-1,1)\), so \(\cC(\bw_q^\triangle)=\{(p_{\mathrm{lat}},q_{\mathrm{lat}})\in\Z^2\mid p_{\mathrm{lat}}\leq0,\ q_{\mathrm{lat}}\geq0\}\), whose primitive generators form a lattice basis; this semigroup is therefore normal.

	Define \(\bv_q^\triangle\) by repeating \(\bv_0=\tp{(-1,0,-1,0,0)}\) in the \(q\) coordinate blocks.
	Every initial monomial of the reduced Gr\"obner basis annihilates \(\bx^{\bv_q^\triangle}\), so \(\bv_q^\triangle\) is a binary fake exponent, and for each \((p_{\mathrm{lat}},q_{\mathrm{lat}})\) the negative support at \((p_{\mathrm{lat}},q_{\mathrm{lat}})\) is the union of the \(q\) indexed copies of the one-block negative support.
	Writing \(J_k^\triangle\) for the union of the \(q\) indexed copies of \(J_k\), we obtain \(\cN_q^\triangle=\{J_0^\triangle,J_1^\triangle,J_2^\triangle\}\) and \(\cN_q^{\triangle,\circ}=\{J_0^\triangle,J_1^\triangle\}\), and the one-block support table shows that \(\cN_q^{\triangle,\circ}\) is ordered but that \(\cN_q^\triangle\) is not.
	Put \(T_j=\prod_{i=1}^qt_{ij}\).
	The definitions and \eqref{eq:lattice-counterexample-MP} give \(e_q^\triangle=T_3\), \(M_{\cN_q^{\triangle,\circ}}=M_{\cN_q^\triangle}=\langle T_3,T_4\rangle\), and \(P_{\cN_q^{\triangle,\circ}}(\bt)=P_{\cN_q^\triangle}(\bt)=\langle T_2T_4,T_3T_5\rangle\).

	We use the map \(\Phi\) that sends \((t_{i1},t_{i2},t_{i3},t_{i4},t_{i5})\) to \((3s_1,\lambda_1,-3\lambda_3,\lambda_3,-\lambda_2)\) on every block.
	Up to nonzero constants, the two monomials \(\bt^{G}\) with \(G\in\cG_{\cN_q^{\triangle,\circ}}\) both map to \(\lambda_3^q\), and the complementary factors of the two monomials \(\bt^{H}\) with \(H\in\cH_{\cN_q^{\triangle,\circ}}\) generate \(\langle\lambda_2^q\rangle\) and \(\langle\lambda_1^q\rangle\), respectively.
	Corollary~\ref{cor:two-support-principal-formula} gives the two identities in \eqref{eq:fixed-rank-ideals} and the graded isomorphism \(\fD_{\cN_q^{\triangle,\circ}}(e_q^\triangle)\simeq S_0/\langle\lambda_2^q,\lambda_3^q\rangle(-q)\).
	Since \(\lambda_2\) and \(\lambda_3\) are linearly independent linear forms, the graded quotient \(S_0/\langle\lambda_2^q,\lambda_3^q\rangle(-q)\) has Hilbert series \(u^q(1-u^q)^2/(1-u)^2=u^q(1+u+\cdots+u^{q-1})^2\) and \(\C\)-dimension \(q^2\).
	The quotient is unchanged by localization at \(\fm\) and completion, and hence \(\fD_{\cN_q^{\triangle,\circ}}(e_q^\triangle)\) has the same \(\C\)-dimension and the same Hilbert series in the sense of \eqref{eq:boundary-completed-Hilbert-definition}.
	Finally, the nonzero rows of \(B_q^\triangle\) form four parallel classes in a connected matroid of rank 2: two rows in one parallel class form a two-element circuit, and two rows in distinct parallel classes, together with a row in a third parallel class, form a three-element circuit.
	The column matroid of \(A_q^\triangle\) is the dual of this matroid and is therefore connected \cite[Chapters~2 and~8]{Oxl11}.
\end{proof}

We summarize the least lattice ranks and the growth of the dimension of the obstruction module.
For ordered negative support families, Corollary~\ref{cor:minimal-lattice-codimension} shows that lattice rank 1 is the least rank in which a nonzero obstruction can occur, and Theorem~\ref{thm:fixed-rank-unbounded} shows that at lattice rank 2 the dimensions of the obstruction modules are unbounded.
For the distinguished collection, Corollary~\ref{cor:minimal-lattice-codimension} shows that the obstruction vanishes in lattice rank at most 1, and Corollary~\ref{cor:cosimple-binary-rank-threshold} shows that a nonzero obstruction forces lattice rank at least 3 when the fake exponent is moving-binary for a generic direction, the direction is nonzero on every nonzero element of \(L\), and the column matroid is cosimple.
The example in Theorem~\ref{thm:three-connected-obstruction} attains this bound with a three-connected column matroid and an affine semigroup \(\cC(\bw)\) generated by a lattice basis.

\begin{remark}
	Corollary~\ref{cor:cosimple-binary-rank-threshold} is restricted to moving-binary fake exponents.
	No reduction from arbitrary integral exponents to the moving-binary case is used.
	Remark~\ref{rem:rank-two-decision} treats every specified fake exponent in lattice rank 2 by means of the matrix \(\Gamma_{\cN_{\bv}}\).
	Proposition~\ref{prop:equal-support-different-obstruction} shows that two fake exponents with the same negative support can have different obstruction modules.
\end{remark}

\section*{Acknowledgments}

Large language models were used in the preparation of this paper.
Codex was used to write the scripts that search for the configuration of Theorem~\ref{thm:ordered-distinguished-counterexample} and that recompute its data independently, and Claude was used to revise the exposition.
Every mathematical statement was checked by the author against the definitions and the cited results, and the author is responsible for the content.
\ifarxiv
The ancillary scripts reproduce selected computations in exact arithmetic in the environment specified in the ancillary README, which describes the scope of the independent checks in SageMath, Macaulay2, and Risa/Asir.
\fi

\end{document}